\documentclass[centertags,leqno,11pt]{amsart}

\usepackage{fullpage}

\usepackage{amsmath}
\usepackage{amssymb}
\usepackage{latexsym}
\usepackage{amsxtra}
\usepackage{amscd}

\usepackage{graphics}
\usepackage{epic}
\usepackage{url}

\newtheorem{thm}{Theorem.}[section]

\newtheorem{lem}[thm]{Lemma.}
\newtheorem{prop}[thm]{Proposition.}

\newtheorem{defn}[thm]{Definition.}

\newtheorem{rem}[thm]{Remark.}

\renewcommand{\em}{\sl}

\newcommand{\Endproof}{\hspace*{\fill} $\Box$ \vspace{1ex} \noindent }

\makeatletter
\renewcommand{\subsection}{\@startsection{subsection}{2}%
{\z@}{-3.25ex plus -1ex minus-.2ex}{-1em}{\bf}} \makeatother

\newcommand{\PP}{\mathbb{P}}
\newcommand{\ZZ}{\mathbb{Z}}
\newcommand{\CC}{\mathbb{C}}

\newcommand{\QQ}{\mathbb{Q}}
\newcommand{\NN}{\mathbb{N}}
\newcommand{\FF}{\mathbb{F}}
\renewcommand{\AA}{\mathbb{A}}
\newcommand{\GG}{\mathbb{G}}
\newcommand{\F}{\mathcal{F}}

\newcommand{\OO}{\mathcal{O}}

\newcommand{\cl}{{\rm cl}}
\newcommand{\C}{{\mathcal C}}
\newcommand{\Irr}{{\rm Irr}}

\newcommand{\si}{{\sigma}}

\newcommand{\MT}{{\rm MT}}

\newcommand{\x}{{\bf x}}
\newcommand{\g}{{\bf g}}

\newcommand{\La}{{\Lambda}}
\newcommand{\GL}{{\rm GL}}
\newcommand{\SL}{{\rm SL}}

\newcommand{\MC}{{\rm MC}}

\newcommand{\LS}{{\rm LS}}

\newcommand{\im}{{\rm im}}
\newcommand{\sym}{{\rm sym}}

\newcommand{\A}{{\mathcal A}}

\newcommand{\SO}{{{\rm SO}}}

\newcommand{\rk}{{\rm rk}}

\newcommand{\J}{{\rm J}}

\newcommand{\Jordan}{{\bf J}}

\renewcommand{\L}{\mathcal{L}}

\newcommand{\n}{{\bf n}}
\newcommand{\ten}{{\otimes}}
\newcommand{\SP}{{\rm SP}}
\newcommand{\Sym}{{\rm Sym}}
\newcommand{\Wr}{{\rm Wr}}

\numberwithin{equation}{section}
\numberwithin{table}{section}
\numberwithin{thm}{section}
\theoremstyle{plain}

\begin{document}

\title{The classification of orthogonally rigid $G_2$-local systems and
  related differential operators}
\author{Michael Dettweiler}
\address[M.~Dettweiler]{
Department of Mathematics\\
University of Bayreuth\\
95440 Bayreuth\\
Germany}
\email{michael.dettweiler@uni-bayreuth.de}

\author{Stefan Reiter}
\address[S.~Reiter]{
Department of Mathematics\\
University of Bayreuth\\
95440 Bayreuth\\
Germany}
\email{stefan.reiter@uni-bayreuth.de}

\begin{abstract} We classify orthogonally rigid local systems of rank $7$ on the punctured 
projective line
whose monodromy 
is dense in the exceptional algebraic group $G_2.$ We obtain differential 
operators corresponding to these local systems under Riemann-Hilbert correspondence.
 \end{abstract}


\maketitle

\section{Introduction} It is well known that the exceptional simple algebraic group $G_2$
can be seen as a subgroup of $\GL_7$ and stabilizes 
the bilinear form $$ x_0^2+x_1y_1+x_2y_2+x_3y_3,$$
where $x_0,x_1,y_1,x_2,y_2,x_3,y_3$ is a suitably chosen basis of the 
underlying $7$-dimen\-sional vector space, cf.~\cite{Asch}. 
It is the aim of this article to classify the 
orthogonally rigid local systems $\L$ of rank~$7$ whose monodromy group is Zariski 
dense in $G_2(\CC)$ and hence leaves the above form invariant.  
{\it Orthogonal rigidity} for an irreducible orthogonally self-dual
complex local system $\L$ on $\PP^1\setminus \{x_1,\ldots,x_{r+1}\}$ of rank $n$ means 
that the following {\rm dimension formula} holds:
\begin{equation}\label{orthrigid}  \sum_{i=1}^{r+1} {\rm codim}(C_{O_n}(g_i))=2\dim(O_n),
\end{equation}
where $C_{O_n}(g_i)$ denotes the 
centralizer of the local monodromy generator $g_i$ in the orthogonal
group $O_n.$ 
The dimension formula \eqref{orthrigid} is equivalent to  the vanishing 
of the parabolic cohomology of $\pi_1(\PP^1\setminus \{x_1,\ldots,x_{r+1}\})$
with values in the Lie algebra of $O_n$ (acting adjointly via the monodromy representation
of $\L$) and
is hence closely related to the dimension of the tangent space of the component
of the space of representations of $\pi_1(\PP^1\setminus \{x_1,\ldots,x_{r+1}\})$
with given local monodromies, cf.~\cite{Weil}. 
The dimension formula is also a necessary condition for the condition that there exist 
 only finitely many equivalence classes of irreducible orthogonally self dual 
local systems $\L$ 
with given local monodromies \cite{SV}. Hence, for such local systems, the
notion of orthogonal rigidity is weaker as the notion of (physical) rigidity 
used in \cite{Katz96} (which can be seen as rigidity relative 
to the larger group $\GL_n$)  but still strong enough to impose a lot of structure  on~$\L.$

By the work of N. Katz on the middle convolution functor
$\MC_\chi,$ all 
rigid irreducible  local systems $\L$ on the punctured line  can be constructed by applying 
iteratively  $\MC_\chi$ and tensor products with rank-$1$-sheaves to a rank-$1$-sheaf.
For orthogonally rigid local systems with $G_2$-monodromy we prove that there
is a similar method of construction: Each such local system can be constructed using 
$\MC_\chi,$ tensor products with rank-$1$-sheaves {\it and}  exceptional 
isomorphisms between small algebraic groups 
which each have a natural interpretation as tensor operations
like  alternating or symmetric  products 
(e.g. $\SO_5=\Lambda^2(\SP_4)$). In some cases, also rational pullbacks are involved.
Our main result is as follows (Thm.~\ref{mainthm}): 

\begin{thm}\label{1.1} Let $\L$ be an orthogonally rigid local system on a punctured projective 
line $\PP^1\setminus 
\{x_1,\ldots,x_{r}\}$  of rank $7$ whose monodromy group
is dense in the exceptional simple group $G_2.$ If $\L$ has nontrivial local 
monodromy at $x_1,\ldots,x_{r},$ then $r=3,4$ and 
$\L$ can be constructed by applying iteratively a 
 sequence of the following operations to a rank-$1$-system:
 \begin{itemize}
 \item Middle convolutions $\MC_\chi,$ with varying $\chi.$ 
 \item Tensor products with rank-$1$-local systems.
 \item Tensor operations like symmetric or alternating products.
 \item Pullbacks along rational functions.
 \end{itemize}
 Especially, each such local system which has quasi-unipotent monodromy 
 is motivic, i.e.,
 arises from the variation of periods of a family of varieties over the punctured
 projective line. 
\end{thm}

A list of the occurring cases together with the local monodromies 
is given in Thm.~\ref{mainthm}.
Rigid local systems on the punctured line with 
$G_2$-monodromy were classified in \cite{DR07}. Since orthogonal 
rigidity for irreducible orthogonal  local systems with $G_2$-monodromy is 
a weaker condition as the usual rigidity condition 
our classification  contains the rigid local systems from \cite{DR07} as special cases.
We remark that the verification that the monodromy group is inside 
the group $G_2$ cannot be decided looking at the local monodromies alone.
To prove this, we make use of recent results of Bogner and Reiter in \cite{ReiterBogner} on the interpretation of 
$\MC_\chi$ at the level of differential operators, related to the 
Hadamard product. Miraculously, the differential operators 
which belong to the local systems of Thm.~\ref{1.1} under Riemann-Hilbert 
correspondence can easily be determined and it can be proven in each case that they
have the property that the second alternating
square has rank $14.$ This implies that the second alternating square of $\L$ 
decomposes into a rank-$14$-factor and a rank-$7$-factor and hence that the monodromy is
contained in the group~$G_2.$ \\

Motivated by the results of Thm.~\ref{1.1} one may ask the question, whether any 
irreducible orthogonally rigid local system can be obtained by a sequence 
of tensor operations, middle convolutions $\MC_\chi,$ and rational pullbacks 
applied to 
a local system of rank one.

\section{Preliminaries on convolution operations}

Recall the construction of the middle convolution from
\cite{Katz96}:  Consider the addition map 
$$ \pi:\AA^1\times \AA^1\to \AA^1,(x,y)\mapsto x+y.$$
Let $\L$ be a complex  valued local system on $\AA^1\setminus \{x_1,\ldots, x_r\}$ 
and let $L=j_*\L[1],$ viewed as perverse sheaf on $\AA^1$   ($j$ denoting the inclusion of 
$\AA^1\setminus \{x_1,\ldots,x_r\}$ into $\AA^1$).  
Let  
 $\L_\chi$ be a local system on $\GG_m,$ defined by a 
nontrivial character $\chi:\pi_1(\GG_m)\to \CC^\times.$ We call $\L_\chi$ a {\it Kummer sheaf}. 
Let 
$L_\chi=(k_*\L_\chi)[1],$ where $k$ denotes the natural inclusion 
of $\GG_m$ to $\AA^1.$ Sometimes we need the following 
variant: using the isomorphism $\AA^1\setminus \{y\}\to 
\GG_m, x\mapsto x-y,$ we can view $\L_\chi$ as local system on $\AA^1\setminus \{y\}.$
This local system is then denoted $\L_{\chi(x-y)}.$ 

 Following Katz \cite{Katz96},
one can define the {\it middle convolution of $\L$ with the Kummer sheaf $\L_\chi$} as
\begin{equation}\label{eqmcchi}  \MC_\chi(\L):=\left(\im(R\pi_!(L\boxtimes L_\chi)\to R\pi_*(L\boxtimes L_\chi)\right)[-1]|_{\AA^1\setminus \{x_1,\ldots,x_r\}}.\end{equation}

\begin{rem}\label{remmc} Since we restrict to $\AA^1\setminus \{x_1,\ldots,x_r\},$  the $0$-th and the $2$-th higher
direct image vanish by the non-triviality of $\L_\chi,$  so \eqref{eqmcchi} is equivalent to 
\begin{equation}\label{eqmcchi2}  \MC_\chi(\L)=\left(\im(R^1\pi_!(j_*\L\boxtimes k_*\L_\chi)\to 
R^1\pi_*(j_*\L\boxtimes k_*L_\chi)\right)|_{\AA^1\setminus \{x_1,\ldots,x_r\}}.\end{equation}
Hence, the middle convoluted local system
$\MC_\chi(\L)$ can be seen as variation of the parabolic cohomology groups 
$H^1(\PP^1, i_*(\L\otimes \L_{\chi(x-y)}))$
over $\AA^1_y\setminus \{x_1,\ldots,x_r\},$
where $i$ is the inclusion 
of $\AA^1\setminus \{x_1,\ldots,x_r,y\}$ into $\PP^1$ and the local systems  
$\L$ and $\L_{\chi(x-y)}$ are viewed as local systems on $\AA^1\setminus 
\{x_1,\ldots,x_r,y\}$ via restriction (cf. \cite{Katz96} and \cite{dw03}).  
\end{rem}

In the usual way we fix a set of generators $\gamma_1,\ldots,\gamma_{r+1}$ 
of $\pi_1(\AA^1\setminus \x),$ where $\gamma_i$ ($i=1,\ldots,r$)  is
a  simple loop which moves counterclockwise 
around $x_i,$ where $\gamma_{r+1}$ moves around $\infty,$
such that 
the product relation 
$ \gamma_1\cdots \gamma_{r+1}=1$ holds. 
Hence, every local system on $\AA^1\setminus \x$ gives, via its monodromy representation
$$\rho_\L:\pi_1(\AA^1\setminus \x, x_0)\to \GL(\L_{x_0})\simeq \GL_n(\CC),$$ rise to 
its monodromy tuple $(A_1,\ldots,A_{r+1}),$
where $A_i=\rho_\L(\gamma_i).$ 
The following result is a consequence of the numerology of the middle convolution
(cf.~\cite[Cor. 3.3.6]{Katz96}):

\begin{lem} \label{lemmonodromy1} Let $\L$ be an irreducible 
local system with 
 monodromy tuple 
$\A=(A_1,\dots,A_{r+1}) \in \GL(V)^{r+1}$, s.t.
at least two $A_i, A_j, 1\leq i<j\leq r$ are non trivial.
  Let $\chi:\pi_1(\GG_m)\to \CC^\times$ 
be the character which sends a counterclockwise generator of $\pi_1(\GG_m)$ 
to $\lambda\in \CC^\times \setminus \{1\}.$ Let $(\tilde{B}_1,\dots,\tilde{B}_{r+1})$
be the monodromy tuple of $\MC_\chi(\L).$ Then the following hold:
\begin{enumerate}

\item The rank $m$ of $\MC_\chi(\L)$ is
      \[ m=\sum_{i=1}^r \rk(A_i-1)+\rk(\lambda^{-1} A_{r+1}-1)-\rk(\L).\]

\item Every  Jordan
block $\J(\alpha,l)$ occurring in the Jordan
decomposition of $A_i$ contributes a Jordan block $\J(\alpha
\lambda,l')$ to the Jordan decomposition of $\tilde{B}_i,$ where
$$ l':\;=\quad
  \begin{cases}
    \quad l, &
                              \quad\text{\rm if $\alpha \not= 1,\lambda^{-1}
$,} \\
    \quad  l-1& \quad \text{\rm if $\alpha =1$,} \\
    \quad l+1, & \quad \text{\rm if $\alpha =\lambda^{-1}$.}
  \end{cases}
  $$
  The only other Jordan blocks which occur in the Jordan
  decomposition of $\
  \tilde{B}_i$ are blocks of the form $\J(1,1).$

\item
Every  Jordan block $\J(\alpha^{-1},l)$ occurring in the Jordan
decomposition of  $A_{r+1}$ contributes a Jordan block $\J(\alpha^{-1}
\lambda^{-1},l')$ to the Jordan decomposition of $\tilde{B}_{r+1},$
where
$$ l':\;=\quad
  \begin{cases}
    \quad l, &
                              \quad\text{\rm if $\alpha \not= 1,\lambda^{-1}
$,} \\
    \quad  l+1& \quad \text{\rm if $\alpha =1$,} \\
    \quad l-1, & \quad \text{\rm if $\alpha =\lambda^{-1}$.}
  \end{cases}
  $$
  The only other Jordan blocks which occur in the Jordan
  decomposition of $\tilde{B}_{r+1}$ are blocks of the form $\J(\lambda^{-1},1).$
\end{enumerate}\end{lem}

By the Riemann-Hilbert correspondence, each local system $\L\in \LS(\AA^1\setminus \x)$
corresponds to an ordinary differential equation (or, equivalently, an 
operator $L$ in the Weyl algebra $\CC[x,\vartheta=x\frac{d}{dx}]$) with regular singularities. 
Let us first describe the tensor operations in the Weyl algebra needed below
(cf.~\cite[Chapter 2]{Put} and \cite{ReiterBogner}).

\begin{defn}
\begin{enumerate}
 \item Let $M_1$, $M_2$ be two differential $\mathbb{C}(x)$-modules. 
The tensor product $\left(M_1\otimes_{\mathbb{C}(x)} M_2,\partial_{M_1\otimes M_2}\right)$ of $M_1$ and $M_2$ over $\mathbb{C}(x)$ is given by the $\mathbb{C}(x)$-vector space $M_1\otimes_{\mathbb{C}(x)} M_2$ together with the derivation
\[\partial_{M_1\otimes M_2}(m_1\otimes m_2):=\partial_{M_1}(m_1)\otimes m_2+m_1\otimes\partial_{M_2}(m_2).\]

\item Let $L_1,L_2\in \mathbb{C}(x)[\partial], \partial=\frac{d}{dx}, $ be two monic differential operators with corresponding differential modules $\left(\left(M_1,\partial_{M_1}\right),\Omega_1\right)$ and $\left(\left(M_2,\partial_{M_2}\right),\Omega_2\right)$. The tensor product $L_1\otimes L_2\in \mathbb{C}(z)[\partial]$ of $L_1$ and $L_2$ over $\mathbb{C}(x)$ is the minimal monic annihilating operator of $\Omega_1\otimes\Omega_2\in M_1\otimes_{\mathbb{C}(x)} M_2$.
\end{enumerate}
\end{defn}

\begin{rem}
\begin{enumerate}

\item By \cite[Corollary 2.19]{Put}, the solution space of $L_1\otimes L_2$ in the Picard-Vessiot field $K\supset \mathbb{C}(x)$ of the operator is spanned by the set \[\{y_1y_2\mid  L_1(y_1)=L_2(y_2)=0\}.\] In particular, $L_1\otimes L_2$ is the monic operator of minimal order, whose solution space is spanned by this set.
\item Symmetric and exterior powers of differential modules and differential operators are defined similarly. If $L\in \mathbb{C}(x)[\partial]$ is monic, by \cite[Corollary 2.23]{Put} and \cite[Corollary 2.28]{Put} $\Sym^2(L)$ is the monic operator of minimal degree whose solution space is spanned by the set
\[\{y_1 y_2\mid L(y_i)=0 \textrm{ for  } i=1,2\}\] and $\Lambda^2(L)$ the monic operator of minimal degree whose solution space
 is spanned by the set of Wronskians
\[\{\Wr(y_1,y_2):=\det\begin{pmatrix}y_1&y_2\\\partial y_1&
\partial y_2
\end{pmatrix}\mid  L(y_i)=0 \textrm{ for } i=1,2\}.\]
\end{enumerate}
\end{rem}

Let $L\in\mathbb{C}[x,\vartheta]$ be Fuchsian, i.e. $L$ has only regular singularities, and smooth on $\AA^1\setminus \{x_1,\ldots,x_r\}$,
let  $f$ be a solution of $L,$  
viewed as section of the local  system $\L$ of solutions of $L,$ and let 
 $a\in\QQ\setminus\ZZ$.
 For two simple loops $\gamma_p, \gamma_q,$ based at $x_0\in \AA^1\setminus \{x_1,\ldots,x_r\},$ and   
 moving counterclockwise around $p,$ resp. $q,$ we define the \textit{Pochhammer contour} \[[\gamma_p,\gamma_q]:=\gamma_p^{-1}\gamma_q^{-1}\gamma_p\gamma_q.\]
 For $y\in \AA^1\setminus \{x_1,\ldots,x_r\},$ the integral 
   \begin{equation}\label{int} C^{p}_a(f)(y):=\int_{[\gamma_{p},\gamma_{y}]}f(x)(y-x)^a\frac{dx}{y-x}
   \end{equation} is called the 
   \textit{convolution} of $f$ and $x^a$ with respect to the Pochhammer contour $[\gamma_{p},\gamma_{y}].$ 
   
   \begin{rem}\label{remmcchi2} If $x^a$ is a local section of   the Kummer sheaf $\L_\chi,$ then 
   the integral $ \int_{[\gamma_{p},\gamma_{y}]}f(x)(y-x)^a\frac{dx}{y-x}$ represents
    an element in $H^1(\AA^1\setminus \{x_1,\ldots,x_r,y\},\L\otimes \L_{\chi(x-y)})$
   in the usual way, cf.~\cite{BlochEsnault} (where we view 
   $\L$ and $\L_\chi$ as local systems on $\AA^1\setminus \{x_1,\ldots,x_r,y\}$ by restriction). 
   Under certain conditions (made explicit in 
   ~\cite{DRFuchsian}),   the analytic continuation 
   of the integral \eqref{int} near the singularities  is in the image 
   of the local monodromy and therefore
 contained in the parabolic cohomology group $H^1(\PP^1,k_*(\L\otimes\L_{\chi(x-y)}))\leq 
 H^1(\AA^1\setminus \{x_1,\ldots,x_r,y\},\L\otimes \L_{\chi(x-y)}),$
 cf.~\cite{dw03}.  By Remark~\ref{remmc}, for varying
$y,$   the integral $C^{p}_a(f)(y)$ can hence be viewed as a section 
of $\MC_\chi(\L).$ 
\end{rem}

  In a similar way as for $C^{p}_a(f)(y),$ define 
\[H^{p}_{a}(f)(y):=\int_{[\gamma_{p},\gamma_{y}]}f(x)\left(1-\frac{y}{x}\right)^{-a}\frac{dx}{x}.\] 
The integral $H^p_a(f)$ 
is called the \textit{Hadamard product} of $f$ and $(1-x)^{-a}$ with respect to the Pochhammer contour $[\gamma_{p},\gamma_{y}]$.
We have the obvious relations
\[C^{p}_a(f)=(-1)^{a-1}H^{p}_{1-a}(x^af),\quad H^{p}_{a}(f)=(-1)^{-a}C^{p}_{1-a}\left(x^{a-1}f\right).\]

In \cite{ReiterBogner}, the following is proved: 

\begin{prop}\label{Falt}
Let $L=\sum_{i=0}^mx^iP_i(\vartheta)\in\mathbb{C}[x,\vartheta]$ be Fuchsian, $f$ a solution of $L$ and $a\in\QQ\setminus\ZZ$.
 Then $C^p_a(f)$ is a solution of
\[\mathcal{C}_a(L):=\sum_{i=0}^my^i\prod_{j=0}^{i-1}(\vartheta+i-a-j)\prod_{k=0}^{m-i-1}(\vartheta-k)P_i(\vartheta-a) \in \mathbb{C}[y,\vartheta]\] for each $p\in\mathbb{P}^1$
and
 $H^p_a(f)$ is a solution of \[\mathcal{H}_a(L):=\sum_{i=0}^my^i\prod_{j=0}^{i-1}(\vartheta+a+j)\prod_{k=0}^{m-i-1}(\vartheta-k)P_i(\vartheta) \in \mathbb{C}[y,\vartheta].\]
\end{prop}

\begin{rem} \label{apositive}

It is shown in \cite[Cor.~4.16]{ReiterBogner} that under some mild restrictions 
the operator $\mathcal{C}_a(L)$ has a right factor
$L\ast_C (\vartheta-a)$ that coincides with the 
differential operator associated to the middle convolution $\MC_\chi(\L)$ via 
the Riemann-Hilbert correspondence (where $\L$ corresponds to $L$ under the
Riemann-Hilbert correspondence). 
Similarly we get the statement for the Hadamard product
$\mathcal{H}_a(L)=\mathcal{C}_{1-a}(L \ten (\vartheta-(a-1)))$ of $L$ with $L_a=(\vartheta-x(\vartheta+a))$.
We denote the irreducible right factor by $L\ast_H L_a$.
It will turn out that  in our situation we can  easily determine the right factor via \cite[Prop.~4.17]{ReiterBogner}.
\end{rem}

\section{Jordan forms in $G_2$ and exceptional isomorphisms}\label{sec3}
Let us collect the information on the
conjugacy classes of the simple algebraic group $G_2.$ Below, we
list the possible Jordan canonical forms 
of elements of the group $G_2(\CC)\leq \GL_7(\CC)$ together with the 
dimensions of the centralizers in the groups $G_2(\CC),\;\SO_7(\CC)$ and in the 
group $\GL_7(\CC).$ 
The list exhausts all possible cases, cf. \cite[Section 1.3]{DR07}.
We use the following conventions: $1_n\in \CC^{n\times n}$
 denotes the identity matrix, 
$\Jordan(n)$ denotes a unipotent 
Jordan 
block of size $n,$ $\omega\in \CC^\times$ 
 denotes a primitive $3$-rd root
of unity,  and $i\in \CC^\times$ denotes a primitive $4$-th root of unity.
Moreover, an expression like $(x\Jordan(2),x^{-1}\Jordan(2),x^2,x^{-2},1)$
denotes a matrix in Jordan canonical form 
in $\GL_7(\CC)$ with one Jordan block of 
size 2 having eigenvalue $x,$ one Jordan block of 
size 2 having eigenvalue $x^{-1},$ and three Jordan blocks of size $1$ 
having eigenvalues $x^2,x^{-2},1$ (resp.).
We also abbreviate a tuple $(x,\ldots,x)$ of length $n$ by $x1_n.$ \\
The Table of $\GL_7$ conjugacy classes of $G_2$ is as follows:
\begin{equation*}
\begin{array}{|c|c|c|}
\hline
 \textrm{Jordan form}  &\mbox{\rm Centralizer dimension in} &  
\textrm{Conditions}\\
               & G_2  \quad \quad \SO_7\quad\quad \GL_7&  \\
\hline 
 (1,1,1,1,1,1,1)& 14 \quad\quad 21 \quad \quad  49&\\
(\Jordan(2),\Jordan(2),1,1,1)& 8\quad\quad 13 \quad \quad 29&\\
 (\Jordan(3),\Jordan(2),\Jordan(2))& 6 \quad\quad 9 \quad \quad 19&\\
(\Jordan(3),\Jordan(3),1)& 4\quad\quad 7 \quad \quad 17&\\
 \Jordan(7)& 2 \quad\quad 3 \quad \quad 7&\\
\hline
(-1,-1,-1,-1,1,1,1)& 6\quad\quad 9 \quad \quad  25 & \\
(-\Jordan(2),-\Jordan(2),1,1,1)&4\quad\quad 7 \quad \quad 17&\\
(-\Jordan(2),-\Jordan(2),\Jordan(3))&4\quad\quad 5 \quad \quad 11&\\
(-\Jordan(3),-1,\Jordan(3))&2\quad\quad 3 \quad \quad 9&\\
\hline
(\omega,\omega,\omega, 1,\omega^{-1},\omega^{-1},\omega^{-1} )& 8\quad\quad 9 \quad \quad 19 & \\
(\omega \Jordan(2),\omega^{-1}\Jordan(2),\omega,\omega^{-1},1)&4\quad\quad 5 \quad \quad 11&\\
(\omega \Jordan(3),\omega^{-1}\Jordan(3),1 )&2\quad\quad 3 \quad \quad 7&\\
\hline 
(i,i,-1,1,i^{-1},i^{-1},-1) &4\quad\quad 5 \quad \quad   13&\\
(i\Jordan(2),i^{-1}\Jordan(2),-1,-1,1)&2\quad\quad 3 \quad \quad 9 &\\
\hline
(x,x,x^{-1},x^{-1},1,1,1)&4 \quad\quad 7 \quad \quad 17& x^2\not=1\\
(x,x,x^2,1,x^{-1},x^{-1},x^{-2})&4 \quad\quad 5 \quad \quad 11&x^4\not=1\not=x^3\\
(x,-1,-x,1,-x^{-1},-1,x^{-1})&2\quad\quad 3 \quad \quad 9& x^4\not=1 \\
(x\Jordan(2),x^{-1}\Jordan(2),x^2,x^{-2},1)&2\quad\quad 3 \quad \quad 7 & x^4\not=1\neq x^3\\
(x\Jordan(2),x^{-1}\Jordan(2),\Jordan(3))&2 \quad\quad 3 \quad \quad 7& x^2\not=1\\
\hline
(x,y,xy,1,(xy)^{-1},y^{-1},x^{-1})&2\quad\quad 3 \quad \quad 7 & \textrm{pairw. diff. eigenvalues}\\
\hline
\end{array}
\end{equation*}

Later, we will need information on how Jordan forms are transformed under the exceptional isomorphisms $\SO_6=\La^2 \SL_4$ and $\SO_5=\La^2 \SP_4.$
Under  the exceptional isomorphism $\SO_6=\La^2 \SL_4$ selected Jordan forms 
are transformed as follows:

\begin{equation*}\label{table2}
\begin{array}{|ccc|}
\hline 
 \textrm{Jordan form in } \SO_6              & \leftrightarrow & \textrm{Jordan form } \Lambda^2\SL_4    \\
\hline 
 (\Jordan(2),\Jordan(2),1,1)&\leftrightarrow & \La^2 (\Jordan(2),1,1)\\
 (\Jordan(5),1)&\leftrightarrow &\La^2 (\Jordan(4))\\
(\Jordan(3),1,-1,-1)&\leftrightarrow &\La^2(i\Jordan(2),-i\Jordan(2)) \\
(\omega \Jordan(3),\omega^{-1}\Jordan(3) )&\leftrightarrow &\La^2(\omega \Jordan(3),1)  \\
     (x\Jordan(2),x^{-1}\Jordan(2),x^2,x^{-2})                          &\leftrightarrow &\La^2 (x\Jordan(2),1,x^{-2})\\
(\Jordan(3),1,x^2,x^{-2})&\leftrightarrow &\La^2(x\Jordan(2),x^{-1}\Jordan(2))  \\
(x,y,xy,(xy)^{-1},y^{-1},x^{-1})&\leftrightarrow &\La^2(x,y,(xy)^{-1},1) \\

\hline
\end{array}
\end{equation*}

Under  the exceptional isomorphism $\SO_5=\La^2 \SP_4$ the Jordan forms 
are transformed as follows:

\begin{equation*}\label{table3}
\begin{array}{|c c c|}
\hline
 \textrm{Jordan form in } \SO_5              & \leftrightarrow & \textrm{Jordan form } \Lambda^2\SP_4    \\
\hline 
 \Jordan(5)& \leftrightarrow &\La^2 (\Jordan(4))\\
(-\Jordan(3),-1,1)& \leftrightarrow &\La^2(-\Jordan(2),\Jordan(2))\\
(\Jordan(3),x^2,x^{-2})& \leftrightarrow &\La^2(x\Jordan(2),x^{-1}\Jordan(2))
 \\
(x\Jordan(2),x^{-1}\Jordan(2),1)& \leftrightarrow &\La^2(\Jordan(2),x,x^{-1})
 \\
(xy,xy^{-1},x^{-1}y,{xy}^{-1},1)& \leftrightarrow &\La^2(x,y,y^{-1},x^{-1}) \\
\hline
\end{array}
\end{equation*}

The proof of the above statements is a straightforward computation using bases and is hence omitted.

\section{The possible cases}

Recall the following result of Scott \cite{Scott77}:

\begin{lem}\label{Scott} Let $K$ be an algebraically closed 
 field and let $V$ be an $n$-dimensional 
$K$-vector space.
   Let $(T_1,\ldots,T_{r+1})\in \GL(V)^{r+1}$ with $T_1\cdots T_{r+1}=1$
 such that $\langle T_1,\ldots, T_{r+1}\rangle$ is irreducible. Then the following statements
 hold:
 \begin{eqnarray*}
   \sum_{i=1}^{r+1} \rk(T_i-1)  \geq  2n && \textit{(Scott Formula)} \\
  \sum_{i=1}^{r+1} \dim (C_{\GL(V)}(T_i)) \leq  (r-1) n^2+2 && \textit{(Dimension count)},
\end{eqnarray*}
 where $\dim (C_{\GL(V)}(T_i))$ denotes the dimension of the centralizer of $T_i$ in 
 $\GL(V)$. 
\end{lem}

Let $(J_1,\ldots,J_{r+1})$ be a tuple of matrices in Jordan form which 
occur in the group $G_2.$ We want to answer the question, if 
there exists a tuple of elements $(T_1,\ldots,T_{r+1})\in G_2(\CC)^{r+1}$ having Jordan 
forms $J_1,\ldots,J_{r+1}$ (resp.) which is the monodromy
tuple of  
an orthogonally rigid local system whose
monodromy group is Zariski dense in the group $G_2(\CC)$ 
(especially, this implies that $T_1 \cdots T_{r+1}=1$ and that the monodromy group
 $\langle T_1,\ldots,T_{r+1}\rangle$ is irreducible). By definition, orthogonal
 rigidity implies that \begin{equation}\label{weilform}
  \sum_{i=1}^{r+1} {\rm codim}(C_{O_7}(T_i))= \sum_{i=1}^{r+1} \left(21-{\rm dim}(C_{O_7}(T_i))\right)=2\dim(O_7)=42.\end{equation}
  Moreover, again by the irreducibility, the tuple $(T_1,\ldots,T_{r+1})$ satisfies 
  the formulae of Lemma~\ref{Scott}. 
 The condition \eqref{weilform} together with the list of possible 
 centralizer dimensions gives the following possibilities for tuples
 of centralizer dimensions $N_i^{O_7}:=\dim(C_{O_7}(T_i)),$
 resp. $N^{\GL_7}_i:=\dim(C_{\GL_7}(T_i))$ (in the case that they contradict the Scott Formula (Lemma~\ref{Scott})   this is noted in 
 the last column by red.):

\begin{center}
\begin{tabular}{|c|c||c| c|} 
\hline 
{\rm case}&$(N_1^{O_7},\ldots,N_{r+1}^{O_7})$&$(N_1^{\GL_7},\ldots,
                  N_{r+1}^{\GL_7})$
                 & remarks\\
\hline
$P_1$ &(13, 5, 3) & (29,  13,  9)&    red. \\
    & &(29, 13, 7) &\\
    & & (29, 11, 9)&\\
    & & (29, 11, 7)&\\
\hline 
$P_2$ &(9, 9, 3)& (25, 25, 9/7)  & red.\\
   &    &    (25, 19, 9) & red.\\
    & & (25, 19, 7) & lin. rigid, s. \cite{DR07}\\
    & & (19, 19, 9) &\\
    & & (19, 19, 7)& \\
\hline 
$P_3$ & (9, 7, 5)& (25,  17, 13/11)&  red.\\
    &  & (19, 17, 13) &\\
    & &(19, 17, 11) &\\
\hline 
$P_4$ &(7, 7, 7) & (17, 17, 17) &  red., s. \cite{DR07}\\
\hline
$P_5$ &(8, 8, 8,  4) & (13,  13, 9, 7)& \\
\hline

\end{tabular}
\vspace{.5cm}

\end{center}

\begin{rem}
The only irreducible subgroups of $G_2(\CC)$ are either finite or the group $\SL_2(\CC)$.
Thus
the Zariski closure of an irreducible subgroup containing an unipotent element
with Jordan form different form $\Jordan(7)$ is $G_2(\CC)$.
\end{rem}

\section{The main result}

\begin{thm} \label{mainthm} Let $\L$ be an orthogonally rigid local system on a punctured projective
line $\PP^1\setminus 
\{x_1,\ldots,x_{r+1}\}$  of rank $7$ whose monodromy group
is dense in the exceptional simple group $G_2.$ If $\L$ has nontrivial local 
monodromy at $x_1,\ldots,x_{r+1},$ then $r=2,3$ and 
$\L$ can be constructed by applying iteratively a 
 sequence of tensor operations, middle convolutions $\MC_\chi,$ and rational pullbacks,
  applied to a local
 system of rank one. Moreover the Jordan canonical forms (under the additional assumptions 
as given in Section~\ref{sec3}) of the local monodromies of $\L$ 
 are as follows:
 \begin{enumerate}
 \item The case $P_1:$ 
{\small
$$
\begin{array}{c|ccc} 
 P_1   &       & &              \\
\hline
 1   &         (  \Jordan(2),  \Jordan(2),1,1,1)& (i,i,-1,1,-1,-i,-i) & \Jordan(7)\\
  2   &         (  \Jordan(2),  \Jordan(2),1,1,1)& (i,i,-1,1,-1,-i,-i) & (\omega \Jordan(3),\omega^{-1}\Jordan(3),1)\\
  3   &           (  \Jordan(2),  \Jordan(2),1,1,1) & (i,i,-1,1,-1,-i,-i)& (x\Jordan(2),x^{-1}\Jordan(2),x^2,x^{-2},1) \\
 4 &           (  \Jordan(2),  \Jordan(2),1,1,1) & (i,i,-1,1,-1,-i,-i)&(x\Jordan(2),x^{-1}\Jordan(2),\Jordan(3))\\
 5&           (  \Jordan(2),  \Jordan(2),1,1,1) &(i,i,-1,1,-1,-i,-i)&(x,y,xy,1,(xy)^{-1},y^{-1},x^{-1})\\
&&&   \pm i \not \in \{ x,y,xy\}\\
6&  (  \Jordan(2),  \Jordan(2),1,1,1) &(-\Jordan(2),-\Jordan(2),\Jordan(3))& \Jordan(7) \\
 7&  (  \Jordan(2),  \Jordan(2),1,1,1) &(-\Jordan(2),-\Jordan(2),\Jordan(3))&(\omega \Jordan(3),\omega^{-1}\Jordan(3),1 )\\
 8&  (  \Jordan(2),  \Jordan(2),1,1,1) &(-\Jordan(2),-\Jordan(2),\Jordan(3))&(x\Jordan(2),x^{-1}\Jordan(2),x^2,x^{-2},1)\\
 9&  (  \Jordan(2),  \Jordan(2),1,1,1) &(-\Jordan(2),-\Jordan(2),\Jordan(3))&(x\Jordan(2),x^{-1}\Jordan(2),\Jordan(3))\\
 10&  (  \Jordan(2),  \Jordan(2),1,1,1) &(-\Jordan(2),-\Jordan(2),\Jordan(3))&(x,y,xy,1,(xy)^{-1},y^{-1},x^{-1})\\
 
11&  (  \Jordan(2),  \Jordan(2),1,1,1) &(z,z,z^{-1},z^{-1},z^2,z^{-2},1)&(-\Jordan(3),-1,\Jordan(3))\\ 
 12&  (  \Jordan(2),  \Jordan(2),1,1,1) &(z,z,z^{-1},z^{-1},z^2,z^{-2},1)&(i\Jordan(2),i^{-1}\Jordan(2),-1,-1,1)\\
 13&  (  \Jordan(2),  \Jordan(2),1,1,1) &(z,z,z^{-1},z^{-1},z^2,z^{-2},1)&(x,-1,-x,1,-x^{-1},-1,x^{-1})\\

14&  (  \Jordan(2),  \Jordan(2),1,1,1) &(z,z,z^{-1},z^{-1},z^2,z^{-2},1)&\Jordan(7)\\
 15&  (  \Jordan(2),  \Jordan(2),1,1,1) &(z,z,z^{-1},z^{-1},z^2,z^{-2},1)&(\omega \Jordan(3),\omega^{-1}\Jordan(3),1 )\\
 16&  (  \Jordan(2),  \Jordan(2),1,1,1)
 &(z,z,z^{-1},z^{-1},z^2,z^{-2},1)&(x\Jordan(2),x^{-1}\Jordan(2),x^2,x^{-2},1)\\
&&&  x\neq z^{\pm 1}\\
 17&  (  \Jordan(2),  \Jordan(2),1,1,1) &(z,z,z^{-1},z^{-1},z^2,z^{-2},1)&(x\Jordan(2),x^{-1}\Jordan(2),\Jordan(3))\\
 &&&x z^{\pm 1}\neq 1 \\ 
 18&  (  \Jordan(2),  \Jordan(2),1,1,1)
 &(z,z,z^{-1},z^{-1},z^2,z^{-2},1)&(x,y,xy,1,(xy)^{-1},y^{-1},x^{-1}) \\
&&&  z \not \in \{ x,y,xy,1,(xy)^{-1},y^{-1},x^{-1} \} \\
19&  (  \Jordan(2),  \Jordan(2),1_3) &(\omega \Jordan(2),\omega^{-1}\Jordan(2),\omega,\omega^{-1},1)&(-\Jordan(3),-1,\Jordan(3))\\
 20&  (  \Jordan(2),  \Jordan(2),1_3) &(\omega \Jordan(2),\omega^{-1}\Jordan(2),\omega,\omega^{-1},1)&(i\Jordan(2),i^{-1}\Jordan(2),-1,-1,1)\\
 21&  (  \Jordan(2),  \Jordan(2),1_3) &(\omega
 \Jordan(2),\omega^{-1}\Jordan(2),\omega,\omega^{-1},1)&(x,-1,-x,1,-x^{-1},-1,x^{-1})\\
&&&   x^6\neq 1\\
22&  (  \Jordan(2),  \Jordan(2),1_3) &(\omega \Jordan(2),\omega^{-1}\Jordan(2),\omega,\omega^{-1},1)&\Jordan(7)\\
 23&  (  \Jordan(2),  \Jordan(2),1_3) &(\omega
 \Jordan(2),\omega^{-1}\Jordan(2),\omega,\omega^{-1},1)&(x\Jordan(2),x^{-1}\Jordan(2),x^2,x^{-2},1)\\
   &&&   x^6\neq 1 \\
 24&  (  \Jordan(2),  \Jordan(2),1_3) &(\omega
 \Jordan(2),\omega^{-1}\Jordan(2),\omega,\omega^{-1},1)&(x\Jordan(2),x^{-1}\Jordan(2),\Jordan(3))\\
&&&   x^3\neq 1\\
 25&  (  \Jordan(2),  \Jordan(2),1_3) &(\omega \Jordan(2),\omega^{-1}\Jordan(2),\omega,\omega^{-1},1)&(x,y,xy,1,(xy)^{-1},y^{-1},x^{-1})\\
&&& x^3\neq 1, y^3\neq 1,(xy)^3 \neq 1\\
\end{array}
$$}

\item The case $P_2:$ 
  The linearly rigid case $P_2(25,19,7)$ is settled in \cite{DR07} and is therefore omitted.

{ \small \[\begin{array}{c|ccccc}
  P_2(19,19,9)   &  &    & & & \#\\
\hline
  2   &      &   (  \Jordan(2),  \Jordan(2),  \Jordan(3))& (  \Jordan(2),  \Jordan(2),  \Jordan(3)) & (i\Jordan(2),i^{-1}\Jordan(2),1,-1_2)&2\\
 3 &        &   (  \Jordan(2),  \Jordan(2),  \Jordan(3)) &(  \Jordan(2),  \Jordan(2),  \Jordan(3)) &(x,-x,-x^{-1},x^{-1},1,-1_2)&4\\
\end{array}\]}
{\small\[ \begin{array}{c|ccccc}
P_2(19,19,7)    &   &    & & &\\
\hline
 1   &      &   (  \Jordan(2),  \Jordan(2),  \Jordan(3))& (  \Jordan(2),  \Jordan(2),  \Jordan(3)) &\Jordan(7)&1\\
  2   &      &   (  \Jordan(2),  \Jordan(2),  \Jordan(3))& (  \Jordan(2),  \Jordan(2),  \Jordan(3)) & (\omega \Jordan(3),\omega^{-1}\Jordan(3),1)&1\\
 3 &        &   (  \Jordan(2),  \Jordan(2),  \Jordan(3)) &(  \Jordan(2),  \Jordan(2),  \Jordan(3)) &(x\Jordan(2),x^{-1 }\Jordan(2),x^2,x^{-2},1)&2\\
 4 &        &   (  \Jordan(2),  \Jordan(2),  \Jordan(3)) &(  \Jordan(2),  \Jordan(2),  \Jordan(3)) &(x\Jordan(2),x^{-1 }\Jordan(2),\Jordan(3))&2 \\
 5 &        &   (  \Jordan(2),  \Jordan(2),  \Jordan(3)) &(  \Jordan(2),  \Jordan(2),  \Jordan(3)) &(x,y,xy,1,(xy)^{-1},y^{-1},x^{-1})&4
 \end{array}\]}

 \item The case $P_3:$ 
  {  \small \[\begin{array}{c|ccccc}
   P_3  &  &    & & &\#\\
\hline

 5&        &   (  \Jordan(2),  \Jordan(2),  \Jordan(3)) &(-\Jordan(2),-\Jordan(2),1_3)&(i,i,-1,1,i^{-1},i^{-1},-1)&2\\
 6&        &   (  \Jordan(2),  \Jordan(2),  \Jordan(3)) &(-\Jordan(2),-\Jordan(2),1_3)&(-\Jordan(2),-\Jordan(2),\Jordan(3))&1\\
 7&        &   (  \Jordan(2),  \Jordan(2),  \Jordan(3)) &(-\Jordan(2),-\Jordan(2),1_3)&(x,x,x^2,1,x^{-1},x^{-1},x^{-2})&2\\
 8&        &   (  \Jordan(2),  \Jordan(2),  \Jordan(3)) &(-\Jordan(2),-\Jordan(2),1_3)&(\omega \Jordan(2),\omega^{-1}\Jordan(2),\omega,\omega^{-1},1)&2\\

 9&        &   (  \Jordan(2),  \Jordan(2),  \Jordan(3)) &(z,z,z^{-1},z^{-1},1_3)&(i,i,-1,1,i^{-1},i^{-1},-1)&2\\
 &&& z^4 \neq 1\\
 10&        &   (  \Jordan(2),  \Jordan(2),  \Jordan(3)) &(z,z,z^{-1},z^{-1},1_3)&(-\Jordan(2),-\Jordan(2),\Jordan(3))&2\\
 11&        &   (  \Jordan(2),  \Jordan(2),  \Jordan(3)) &(z,z,z^{-1},z^{-1},1_3)&(x,x,x^2,1,x^{-1},x^{-1},x^{-2})&2\\
 &&& x^{\pm 3} z^{\pm 1}\neq 1 ,x^{\pm 1} z^{\pm 1}\neq 1\\
11&        &   (  \Jordan(2),  \Jordan(2),  \Jordan(3)) &(z,z,z^{-1},z^{-1},1_3)&(x,x,x^2,1,x^{-1},x^{-1},x^{-2})&1\\
&&& x^{\pm 3} z^{\pm 1}= 1 ,x^{\pm 1} z^{\pm 1}= 1\\
 12&        &   (  \Jordan(2),  \Jordan(2),  \Jordan(3)) &(z,z,z^{-1},z^{-1},1_3)&(\omega \Jordan(2),\omega^{-1}\Jordan(2),\omega,\omega^{-1},1)&2\\
&&& z^{3}\neq 1 \\
12&        &   (  \Jordan(2),  \Jordan(2),  \Jordan(3)) &(z,z,z^{-1},z^{-1},1_3)&(\omega \Jordan(2),\omega^{-1}\Jordan(2),\omega,\omega^{-1},1)&1\\
&&& z^{3}= 1 \\
\end{array}\]}

\item The case $P_5:$ 
{\small \[ \begin{array}{c|cccccc}
 P_5   &   &    & & &&\#\\
\hline
1    &      &   (  \Jordan(2),  \Jordan(2),1_3)&  (  \Jordan(2),
\Jordan(2),1_3)& (\omega 1_3,1,\omega^{-1} 1_3) & (\Jordan(3),\Jordan(3),1)&1\\
 2   &      &   (  \Jordan(2),  \Jordan(2),1_3)&  (  \Jordan(2),
 \Jordan(2),1_3)&  (\omega 1_3,1,\omega^{-1} 1_3)  & (-\Jordan(2),-\Jordan(2),1_3)&1\\
3    &      &   (  \Jordan(2),  \Jordan(2),1_3)&  (  \Jordan(2),
\Jordan(2),1_3)&  (\omega 1_3,1,\omega^{-1} 1_3) & (x,x,x^{-1},x^{-1},1_3)&1\\
  \end{array} \]}
\end{enumerate}
For each tuple of Jordan forms in the above list, there exists a local system $\L$ 
of rank $7$ whose monodromy group
is dense in the exceptional simple group $G_2.$ 
The cardinality of equivalence classes with given local monodromies 
under the diagonal conjugation in $G_2$ is listed under $\#.$  \end{thm}

\proof The proof  is divided into the cases $P_1,P_2,P_3,P_5,$ where
each case is dealt with in one of the   following subsections:

\subsection{The case $P_1$} 
We can assume $x_1=0, x_2=1,x_3=\infty.$
If $\phi,\phi':\pi_1(\GG_m)\to \CC^\times$ are characters, there exists a unique
local system $\L(\phi,\phi')$ of rank one on $\AA^1\setminus \{0,1\}$ whose local monodromies at $0,1$ is
$\phi,\phi',$ resp. The functor, which sends a local system $\L$ 
on $\AA^1\setminus\{0,1\}$ to $\L\otimes \L(\phi,\phi')$ is denoted 
$\MT_{\L(\phi,\phi')}.$

The irreducibility condition 
and the Scott Formula (Lemma~\ref{Scott}) imply that only the possibilities listed in the case~$P_1$ of 
Thm.~\ref{mainthm} occur. At first we prove that in each of the listed cases, 
there exists an  orthogonally rigid local system
that can be reduced to a rank $1$ system via the middle
convolution and tensor products. 
Applying the functor
\[ M_\phi= \MC_\phi \circ \MT_{\L(1,\phi^{-1})}  \circ \MC_{\phi^{-1}}  \circ \MT_{\L(1,\phi)},\]
where 
\[ \phi=\left\{ \begin{array}{ccc}
                    i & 1)-5) \\
                     -1& 6)-10)\\
                    z & 11)-18)\\
                    \omega& 19)-26)
                \end{array} \right.,\]
we obtain an orthogonally rigid local system of rank $5$ or $6$ by \cite[Thm. 2.4.(i)]{DRFuchsian} and \cite[Cor. 5.15]{dr00}.
The change of the Jordan form of the local monodromies in each step
can be traced via Lemma~\ref{lemmonodromy1}.
The result is
{ \small \[\begin{array}{c|c|cccc}
  nr.   & \rk  &    & & &\\
\hline
  1  &  5   &   (  \Jordan(2),  \Jordan(2),1)& (  \Jordan(3),-1,-1) & \Jordan(5) &\\
  2   &  6   &   (  \Jordan(2),  \Jordan(2),1,1)& (  \Jordan(3),1,-1,-1) & (\omega \Jordan(3),\omega^{-1}\Jordan(3))&\\
  3   &    6   &   (  \Jordan(2),  \Jordan(2),1,1) & (  \Jordan(3),1,-1,-1)& (x\Jordan(2),x^{-1}\Jordan(2),x^2,x^{-2})&\\ 
 4 &    5   &   (  \Jordan(2),  \Jordan(2),1) & (  \Jordan(3),-1,-1)&(x\Jordan(2),x^{-1}\Jordan(2),1)&\\
 5&    6   &   (  \Jordan(2),  \Jordan(2),1,1) & (  \Jordan(3),1,-1,-1)&(x,y,xy,(xy)^{-1},y^{-1},x^{-1})&\\
 6&    5   &   (  \Jordan(2),  \Jordan(2),1) &  \Jordan(5)&\Jordan(5) &\\
 7&    6   &   (  \Jordan(2),  \Jordan(2),1,1) &  (\Jordan(5),1)&(\omega \Jordan(3),\omega^{-1}\Jordan(3) )&\\
 8&    6   &   (  \Jordan(2),  \Jordan(2),1,1) &(  \Jordan(5),1)&(x\Jordan(2),x^{-1}\Jordan(2),x^2,x^{-2})&\\
 9&    5   &   (  \Jordan(2),  \Jordan(2),1) &  \Jordan(5)&(x\Jordan(2),x^{-1}\Jordan(2),1)&\\
 10&    6   &   (  \Jordan(2),  \Jordan(2),1,1) &(  \Jordan(5),1)&(x,y,xy,(xy)^{-1},y^{-1},x^{-1})&\\
 \end{array}\]}
{ \small \[\begin{array}{c|c|cccc} 
11&    5   &   (  \Jordan(2),  \Jordan(2),1) &(  \Jordan(3),z^2,z^{-2})&(-\Jordan(3),-1,1)&\\
12&    6   &   (  \Jordan(2),  \Jordan(2),1,1) &(  \Jordan(3),1,z^2,z^{-2})&(i\Jordan(2),i^{-1}\Jordan(2),-1,-1)&\\
 13&    6   &   (  \Jordan(2),  \Jordan(2),1,1) &(  \Jordan(3),1,z^2,z^{-2})&(x,-1,-x,-x^{-1},-1,x^{-1})&\\
14&    5   &   (  \Jordan(2),  \Jordan(2),1) &(  \Jordan(3),1,z^2,z^{-2})&\Jordan(5)&\\
 15&    6   &   (  \Jordan(2),  \Jordan(2),1,1) &(  \Jordan(3),1,z^2,z^{-2})&(\omega \Jordan(3),\omega^{-1}\Jordan(3) )&\\
 16&    6   &   (  \Jordan(2),  \Jordan(2),1,1) &(  \Jordan(3),1,z^2,z^{-2})&(x\Jordan(2),x^{-1}\Jordan(2),x^2,x^{-2})&\\
 17&    5   &   (  \Jordan(2),  \Jordan(2),1) &(  \Jordan(3),z^2,z^{-2})&(x\Jordan(2),x^{-1}\Jordan(2),1)&\\
 18&    6   &   (  \Jordan(2),  \Jordan(2),1,1) &(  \Jordan(3),1,z^2,z^{-2})&(x,y,xy,(xy)^{-1},y^{-1},x^{-1})&\\
 19&    5   &   (  \Jordan(2),  \Jordan(2),1) &(  \Jordan(3),\omega,\omega^{-1})&(-\Jordan(3),-1,1)&\\
 20&    6   &   (  \Jordan(2),  \Jordan(2),1,1) &(  \Jordan(3),1,\omega,\omega^{-1})&(i\Jordan(2),i^{-1}\Jordan(2),-1,-1)&\\
 21&    6   &   (  \Jordan(2),  \Jordan(2),1,1) &(  \Jordan(3),1,\omega,\omega^{-1})&(x,-1,-x,-x^{-1},-1,x^{-1})&\\
 22 & 5   &   (  \Jordan(2),  \Jordan(2),1) &(  \Jordan(3),\omega,\omega^{-1})&\Jordan(5)&\\
 23&    6   &   (  \Jordan(2),  \Jordan(2),1,1) &(  \Jordan(3),1,\omega,\omega^{-1})&(x\Jordan(2),x^{-1}\Jordan(2),x^2,x^{-2})&\\
 24&    5   &   (  \Jordan(2),  \Jordan(2),1) &(  \Jordan(3),\omega,\omega^{-1})&(x\Jordan(2),x^{-1}\Jordan(2),1)&\\
 25&    6   &   (  \Jordan(2),  \Jordan(2),1,1) &(  \Jordan(3),1,\omega,\omega^{-1})&(x,y,xy,(xy)^{-1},y^{-1},x^{-1})&\\
\end{array}\]}

Using tensor identities $\Lambda^2 \SL_4 = \SO_6$ and $\Lambda^2 \SP_4 = \SO_5$ 
and the effect on the Jordan forms listed in Section~\ref{sec3}, we obtain up to two linearly rigid local systems of rank $4$.
(linear rigidity is  synonymously used as physical rigidity in \cite{Katz96} and, in the above terms, it is the same as $\GL_n$-rigidity)
In some cases we get only one local system due to the Scott Formula.
Their  local monodromies in Jordan form is as follows.
{\small \[\begin{array}{c|c|cccc}
  nr.   & \rk  &    & & &\\
\hline
  1 &  4   &   (  \Jordan(2),1,1)& (i\Jordan(2),-i\Jordan(2)) & \Jordan(4)\\
  2   &  4   &   (  \Jordan(2),1,1)& (i\Jordan(2),-i\Jordan(2)) & (\omega \Jordan(2),\omega,1)\\
  3   &    4   &   (  \Jordan(2),1,1) & (i\Jordan(2),-i\Jordan(2))& (x\Jordan(2),1,x^{-2})&\\ 
 4 &    4   &   (  \Jordan(2),1,1) & (i\Jordan(2),-i\Jordan(2))&(  \Jordan(2),x,x^{-1})&\\
 5&    4   &   (  \Jordan(2),1,1) & (i\Jordan(2),-i\Jordan(2))&(x,y,(xy)^{-1},1)&\\
6&    4   &   (  \Jordan(2),1,1) &  \Jordan(4)& -\Jordan(4)\\
 7&    4   &   (  \Jordan(2),1,1) &  \Jordan(4)&(-\omega \Jordan(3),-1 )\\
 8&    4   &   (  \Jordan(2),1,1) &  \Jordan(4)&(-x\Jordan(2),-1,-x^{-2})\\
 9&    4   &   (  \Jordan(2),1,1) &  \Jordan(4)&(  -\Jordan(2),-x,-x^{-1})\\
 10&    4   &   (  \Jordan(2),1,1) &  \Jordan(4)&(-x,-y,-(xy)^{-1},-1)\\
  11&    4   &   (  \Jordan(2),1,1) &(z  \Jordan(2),z^{-1}  \Jordan(2))&(-  \Jordan(2),  \Jordan(2))\\
 12&    4   &   (  \Jordan(2),1,1) &(z  \Jordan(2),z^{-1}  \Jordan(2))& \pm (i\Jordan(2),1,-1)\\
 13&    4   &   (  \Jordan(2),1,1) &(z  \Jordan(2),z^{-1}  \Jordan(2))& \pm (x,-x^{-1},1,-1)\\
\end{array}\]}
{\small \[\begin{array}{c|c|cccc}
 14 &    4   &   (  \Jordan(2),1,1) &(z  \Jordan(2),z^{-1}  \Jordan(2))& \pm \Jordan(4)\\
15&    4   &   (  \Jordan(2),1,1) &(z  \Jordan(2),z^{-1}  \Jordan(2))& (\omega \Jordan(3),1)\\
  &    4   &   (  \Jordan(2),1,1) &(z  \Jordan(2),z^{-1}  \Jordan(2))& -(\omega \Jordan(3),1)\\    
 &&&&z^6\neq 3\\ 

 16&    4   &   (  \Jordan(2),1,1) &(z  \Jordan(2),z^{-1}  \Jordan(2))&(x\Jordan(2),1,x^{-2})\\
&    4   &   (  \Jordan(2),1,1) &(z  \Jordan(2),z^{-1}  \Jordan(2))&-(x\Jordan(2),1,x^{-2})\\
 &&&& -zx^{\pm 2} \neq 1,-zx^{\pm 1} \neq 1\\
 17&    4   &   (  \Jordan(2),1,1) &(z  \Jordan(2),z^{-1}  \Jordan(2))&(  \Jordan(2),x,x^{-1})\\
 & 4   &   (  \Jordan(2),1,1) &(z  \Jordan(2),z^{-1}  \Jordan(2))&-(  \Jordan(2),x,x^{-1})\\
 &&&&-x z^{\pm 1}\neq 1 \\ 
 18&    4   &   (  \Jordan(2),1,1) &(z  \Jordan(2),z^{-1}  \Jordan(2))&(x,y,(xy)^{-1},1)\\
 & 4   &   (  \Jordan(2),1,1) &(z  \Jordan(2),z^{-1}  \Jordan(2))&-(x,y,(xy)^{-1},1)\\
   &&&&-z^{\pm 1}\not \in \{ x,y,xy\}   \\
 19&    4   &   (  \Jordan(2),1,1)& (\omega\Jordan(2),\omega^{-1}\Jordan(2))&(  \Jordan(2),-  \Jordan(2))\\
 20&    4   &   (  \Jordan(2),1,1)& (\omega\Jordan(2),\omega^{-1}\Jordan(2))&\pm (i\Jordan(2),1,-1)\\
 21&    4   &   (  \Jordan(2),1,1)& (\omega\Jordan(2),\omega^{-1}\Jordan(2))&\pm (x,-x^{-1},1,-1)\\
22&    4   &   (  \Jordan(2),1,1)& (\omega\Jordan(2),\omega^{-1}\Jordan(2))& \pm \Jordan(4)\\
 23&    4   &   (  \Jordan(2),1,1)& (\omega\Jordan(2),\omega^{-1}\Jordan(2))&(x\Jordan(2),1,x^{-2})\\
&    4   &   (  \Jordan(2),1,1)& (\omega\Jordan(2),\omega^{-1}\Jordan(2))&-(x\Jordan(2),1,x^{-2})\\
  &&&& x^{12}\neq 1\\
 24&    4   &   (  \Jordan(2),1,1)& (\omega\Jordan(2),\omega^{-1}\Jordan(2))&(  \Jordan(2),x,x^{-1}) &\\
&    4   &   (  \Jordan(2),1,1)& (\omega\Jordan(2),\omega^{-1}\Jordan(2))&-(  \Jordan(2),x,x^{-1}) &\\
&&&&x^6\neq 1\\
 25&    4   &   (  \Jordan(2),1,1)& (\omega\Jordan(2),\omega^{-1}\Jordan(2))&(x,y,(xy)^{-1},1)\\
&    4   &   (  \Jordan(2),1,1)& (\omega\Jordan(2),\omega^{-1}\Jordan(2))&-(x,y,(xy)^{-1},1)\\
 &&&&x^6\neq 1, y^6\neq 1 ,(xy)^6\neq 1
\end{array}\]}

\begin{rem}
 These tuples of local monodromies in Jordan form arise from  hypergeometric irreducible
 local systems. This can be checked by \cite{Beukers-Heckman} or the Katz Existence
 algorithm. 
\end{rem}

>From the discussion above we know that there exist at most $2$
orthogonally rigid
local systems having $G_2$-monodromy with the same local monodromies.
We reduce the  monodromy tuple modulo $l$ 
to order to show
that there exists at most $1$ such local system.
Via Prop.~\ref{Propstructconst} we can compute the normalized structure constant $n(\cl(\si_1),\ldots,\cl(\si_{r+1}))$
(the definition is recalled in the Appendix) 
corresponding to the
reduced  monodromy tuple $(\si_1,\dots,\si_{r+1})$
via the generic character table of the group $G_2(q)$.
Using CHEVIE, \cite{CHEVIE}, we obtain in the notation of Chang and Ree
(cf. Rem.~\ref{ChaRee}) the following list.
(Note that the output of CHEVIE comes along with a list of possible exceptions
 depending on the eigenvalues of the conjugacy classes.
 In most of the cases these exceptions correspond to ours obtained from the Scott Formula.
 In the remaining cases one can proceed as follows.
 The characters of $G_2(q)$ fall into finitely many families $\F_j$. 
 Using the character table of $G:=G_2(q)$ in CHEVIE or \cite[Anhang B]{Hiss}, 
 one easily sees that the contribution of most those families 
  \[ 
    \left\vert  \frac{\mid G\mid ^{r-1}}{\prod_i \mid C_G(\si_i) \mid}
     \sum_{\chi \in \F_j} \frac{\prod_i \chi(\si_i)} {\chi(1)^{r-1}}\right\vert \] 
to the normalized
 structure constant is bounded by $c/q,$ where $c$ is constant.
 Thus we finally get that 
 $\lim _k (\lfloor n(\C(q^k))\rfloor ) <2.$)

\[\begin{tabular}{ccc} 

$\begin{array}{|ccc|} 
\hline
n(u_2,k_{2,2},u_6)&=&1\\
n(u_2,k_{2,2},k_{2,3})&=&2-\frac{3}{2q}\\
n(u_2,k_{2,2},k_{24})&=&2-\frac{1}{2q}\\
n(u_2,k_{2,2},k_{3,2})&=&1\\
n(u_2,k_{2,2},k_{3,3,1})&=&0\\  
n(u_2,k_{2,2},k_{3,3,2})&=&0\\  
n(u_2,k_{2,2},h_{1a,1})&=&1\\  
n(u_2,k_{2,2},h_{1b,1})&=&1\\  
n(u_2,k_{2,2},h_{1})&=&1\\  
\hline
\end{array}
\begin{array}{|ccc|} 
\hline
n(u_2,k_{3,1},u_6)&=&1\\
n(u_2,k_{3,1},k_{2,3})&=&1\\
n(u_2,k_{3,1},k_{24})&=&0\\
 n(u_2,k_{3,1},k_{3,2})&<&1\\
n(u_2,k_{3,1},k_{3,3,1})&<&1\\  
n(u_2,k_{3,1},k_{3,3,2})&<&1\\ 
n(u_2,k_{3,1},h_{1a,1})&=&1\\  
n(u_2,k_{3,1},h_{1b,1})&=&1\\ 
n(u_2,k_{3,1},h_{1})&=&1\\  
\hline
\end{array}
\begin{array}{|ccc|} 
\hline
n(u_2,h_{1b},u_6)&=&1\\
n(u_2,h_{1b},k_{2,3})&=&1\\
n(u_2,h_{1b},k_{24})&=&0\\
n(u_2,h_{1b},k_{3,2})&=&1\\
n(u_2,h_{1b},k_{3,3,1})&=&0\\  
n(u_2,h_{1b},k_{3,3,2})&=&0\\  
n(u_2,h_{1b},h_{1a,1})&=&1\\  
n(u_2,h_{1b},h_{1b,1})&=&1\\  
n(u_2,h_{1b},h_{1})&=&1\\  
\hline
\end{array}$
\end{tabular}
\]

By Thm.~\ref{nC} we have hence at most one 
orthogonally rigid
local system having $G_2$-monodromy with given local monodromies.

To show the existence, we construct a  differential operator
by translating the middle convolution operations and tensor product operations
 to the level of differential operators, cf. Remark~\ref{apositive}.  
To simplify the construction we rather work with the middle Hadamard product than
with the middle convolution.
Let $L=\vartheta(\vartheta-c)(\vartheta-d)(\vartheta+(c+d))-x(\vartheta+a)^2(\vartheta+1-a)^2$ 
be hypergeometric with Riemann scheme
{\small \begin{eqnarray*} &&{\mathcal{R}}(L)=\left\{\begin{array}{ccc}
                             0 & 1 & \infty \\
                             \hline
                            0&0&a\\
                            c&1&a\\
                            d&1&1-a\\
                           -c-d&2&1-a
                    \end{array}\right\}, \\
\end{eqnarray*}}
and $L_b=\vartheta-x(\vartheta+b),\; b \in \{a,1-a\}$.
Thus we get the formally self adjoint operator
\[ P_1:=L_a\ast_H L_{1-a} \ast_H  \Lambda^2(L), \]
\[\begin{array}{ccl}
 P_1&=&\vartheta \left( \vartheta -d \right)  \left( \vartheta +d \right)  \left( \vartheta -c \right) 
 \left( \vartheta +c \right)  \left( c+d+\vartheta  \right)  \left( -c-d+\vartheta  \right)-\\
 &&x
 \left( 2\,\vartheta +1 \right)  \left( \vartheta +a \right)  \left( \vartheta +1-a \right)\cdot
( {\vartheta }^{4}+2\,{\vartheta }^{3}+
(2-{c}^{2}-{d}^{2}-2\,{a}^{2}+2\,a-cd)\,{\vartheta }^{2}+\\
&&(1-{c}^{2}-{d}^{2}-2\,{a}^{2}+2\,a-cd)\,{\vartheta }-2\,a \left( a-1 \right)  \left( {a}^{2}-a-cd-{d}^{2}-{c}^{2}+1\right) )+ \\
&&
 x^2\left( \vartheta +1 \right)  \left( \vartheta +2-2\,a \right)  \left( \vartheta +1+a \right) 
 \left( \vartheta +a \right)  \left( \vartheta +2-a \right)  \left( \vartheta +1-a \right) 
 \left( \vartheta +2\,a \right) 
\end{array}\]

with Riemann scheme
{\small \[ {\mathcal{R}}(P_1)=\left\{\begin{array}{ccc}
                             0 & 1 & \infty \\
                             \hline
                             0 & 0 & 1 \\
                             c & 1& a\\
                             d& 1& a+1\\
                             c+d &2& 2a \\
                             -c-d&3& 2-2a\\
                             -d&3& 1-a\\
                             -c&4& 2-a
                        \end{array}
                     \right\}.\]}
Specializing the parameters $a,c,d$ suitably we get a differential equation
for all the $P_1$-cases if we can show that it has $G_2$-monodromy.
Using MAPLE one gets that $\Lambda^2(L)$ has degree $14$.
Thus $\Lambda^2$ of the monodromy representation decomposes into
a rank $7$ and a rank $14$ representation.
Hence, by Table 5 in \cite{OnishchikVinberg} the monodromy group is contained in $G_2$.
Thus the orthogonally rigid local systems
with $G_2$-monodromy in the case $P_1$ are uniquely determined by their
local monodromy.

\subsection{The case $P_2$}
In the case $P_2(19,19,9)$  the list of possible Jordan forms of the local monodromies
is
{ \small \[ \begin{array}{c|c|cccc}
  nr.   & \rk  &    & & &\\
\hline
 1   &  7   &   (  \Jordan(2),  \Jordan(2),  \Jordan(3))& (  \Jordan(2),  \Jordan(2),  \Jordan(3)) &(-\Jordan(3),-1,\Jordan(3)) &\\
  2   &  7   &   (  \Jordan(2),  \Jordan(2),  \Jordan(3))& (  \Jordan(2),  \Jordan(2),  \Jordan(3)) & (i\Jordan(2),i^{-1}\Jordan(2),-1,-1,1)\\
 3 &    7   &   (  \Jordan(2),  \Jordan(2),  \Jordan(3)) &(  \Jordan(2),  \Jordan(2),  \Jordan(3)) &(x,-x,-x^{-1},x^{-1},1,-1,-1)
\end{array}\]}

Applying the functor
\[ M_\phi= \MT_{\L(-1,-1)} \circ \MC_{-1} \circ \MT_{\L(-1,-1)} \circ \MC_{-1} ,\]
we obtain in each case an orthogonally rigid local system of rank $5$  with the following tuple of Jordan forms
{ \small \[\begin{array}{c|c|cccc}
  nr.   & \rk  &    & & &\\
\hline
 1   &  5   &   (-  \Jordan(2),-  \Jordan(2),1)& (  \Jordan(2),-  \Jordan(2),1) &(-\Jordan(3),\Jordan(3)) & red.\\
  2   &  5   &   (-  \Jordan(2),-  \Jordan(2),1)& (-  \Jordan(2),-  \Jordan(2),1) & (i\Jordan(2),i^{-1}\Jordan(2),1)\\
 3 &    5   &   (-  \Jordan(2),-  \Jordan(2),1) &(-  \Jordan(2),-  \Jordan(2),1) &(x,-x,-x^{-1},x^{-1},1)
\end{array}\]}
Using the isomorphism 
\[ \La^2 \SP_4 \cong \SO_5 \]
we get
{\small \[  \begin{array}{c|c|cccc}
  nr.   & \rk  &    & & &\\
\hline
  2   &  4   &   (-  \Jordan(2),1,1)& (-  \Jordan(2),1,1) & -(\Jordan(2),i,i^{-1})\\
 3 &    4   &   (-  \Jordan(2),1,1) &(-  \Jordan(2),1,1) &\pm (ix,i,-i,(ix)^{-1})
\end{array}\]}
Applying $\MC_{-1}$ we get an orthogonally rigid local system of rank $3, 4$ resp.,  with the following tuple of Jordan forms
of the local monodromies.

{ \small \[\begin{array}{c|c|cccc}
  nr.   & \rk  &    & & &\\
\hline
  2   &  3   &  \Jordan(3)&\Jordan(3) & (1,i,i^{-1})\\
 3 &    4   &   (  \Jordan(3),1) &(  \Jordan(3),1) &\mp (ix,i,-i,(ix)^{-1})
\end{array}\]}

Via the isomorphisms 
\[ \sym^2 \SP_2 =\SO_3, \quad \SL_2\otimes \SL_2 = \SO_4 \]
we can decompose it into linearly rigid irreducible 
local systems $\L_1$ and $\L_2$ of rank $2$ with the following tuple of Jordan forms.

{ \small \[\begin{array}{c|c|cccc}
  nr.   & \L_i  &    & & &\\
\hline
  2   &  \L_1=\L_2  &  \Jordan(2)&\Jordan(2) & \pm (\zeta_8,\zeta_8^{-1})\\
 3 &    \L_1   &\Jordan(2) &\Jordan(2) & (y,y^{-1}) &y^2=x \\
   &    \L_2   &\Jordan(2) &\Jordan(2) &\pm (iy,i^{-1}y^{-1})
\end{array}\]}

\begin{rem}
The cases $(2)$ and$ (3) $ are 
quadratic pullbacks
of linearly rigid local systems (cf. \cite{DR07})
with  $G_2$-monodromy $(A,B,C), ABC=1, A^2=1,$
via \[  (A,B,C) \mapsto (A^2,B^{A},B,C^2).\]
{ \small\[ \begin{array}{c|c|cccc}
  nr.   & \rk  &    & & &\\
\hline
  2   &  7   & (-1_4,1_3)&  (  \Jordan(2),  \Jordan(2),  \Jordan(3))& (\zeta_8 \Jordan(2),\zeta_8^{-1}\Jordan(2),\zeta_8^2,\zeta_8^{-2},1)\\
 &     & (-1_4,1_3)&  (  \Jordan(2),  \Jordan(2),  \Jordan(3))& (-\zeta_8 \Jordan(2),-\zeta_8^{-1}\Jordan(2),\zeta_8^2,\zeta_8^{-2},1)\\
 3 &    7   & (-1_4, 1_3) &(\Jordan(2),  \Jordan(2),  \Jordan(3)) &(y,iy,iy^{-1},y^{-1},1,i,-i)& y^2=x\\
& & (-1_4,1_3) &(
 \Jordan(2),  \Jordan(2),  \Jordan(3)) &(y,-iy,-iy^{-1},y^{-1},1,i,-i)& y^2=x
\end{array}\]}

This shows the existence of two, four resp., orthogonally rigid local systems with $G_2$ monodromy  and the same
local monodromy in the case (2), (3) resp.. 
\end{rem}

In the case $P_2(19,19,7)$  the list of possible Jordan forms of the local monodromies is as follows:

{ \small \[\begin{array}{c|c|cccc}
  nr.   & \rk  &    & & &\\
\hline
 1   &  7   &   (  \Jordan(2),  \Jordan(2),  \Jordan(3))& (  \Jordan(2),  \Jordan(2),  \Jordan(3)) &\Jordan(7)&\\
  2   &  7   &   (  \Jordan(2),  \Jordan(2),  \Jordan(3))& (  \Jordan(2),  \Jordan(2),  \Jordan(3)) & (\omega \Jordan(3),\omega^{-1}\Jordan(3),1)\\
 3 &    7   &   (  \Jordan(2),  \Jordan(2),  \Jordan(3)) &(  \Jordan(2),
 \Jordan(2),  \Jordan(3)) &(x\Jordan(2),x^{-1 }\Jordan(2),x^2,x^{-2},1)\\
 4 &    7   &   (  \Jordan(2),  \Jordan(2),  \Jordan(3)) &(  \Jordan(2),  \Jordan(2),  \Jordan(3)) &(x\Jordan(2),x^{-1 }\Jordan(2),\Jordan(3)) \\
 5 &    7   &   (  \Jordan(2),  \Jordan(2),  \Jordan(3)) &(  \Jordan(2),  \Jordan(2),  \Jordan(3)) &(x,y,xy,1,(xy)^{-1},y^{-1},x^{-1})
\end{array}\]}

The corresponding normalized structure constant of the reduced monodromy tuple
can again be computed via Lemma~\ref{Propstructconst} and CHEVIE, cf. \cite{CHEVIE}: 
\[ \begin{array}{clcccc}
 1) &    \n(u_2,u_2,u_6) = (3q-2)/q & \\
 2) &\n(u_2,u_2,k_{3,2}) = \frac{4q-1}{3q},\quad   \n(u_2,u_2,k_{3,3,i}) =\frac{q-1}{3q}, \quad i=1,2 \\ 
3)&\n(u_2,u_2,h_{1a,1}) = \left\{\begin{array}{cc} 
                             3(q-1)/q, & h_{1a,1}=h_{1a,1}'^2 \\
                              (q-1)/q, & h_{1a,1} \neq h_{1a,1}'^2 
                         \end{array}\right.\\
4)&\n(u_2,u_2,h_{1b,1}) = \left\{\begin{array}{cc} 
                             (3q-1)/q, & h_{1b,1}=h_{1b,1}'^2 \\
                              (q-1)/q, & h_{1b,1} \neq h_{1b,1}'^2 
                         \end{array}\right. \\ 
5)& \n(u_2,u_2,h_{1})  = \left\{\begin{array}{cc} 
                             4 & h_{1}=h_{1}'^2 \\
                              0 & h_{1} \neq h_{1}'^2 
                         \end{array}\right.   
\end{array}\]

Linearly rigid irreducible triples $(A,B,C), A^2=1,$ in the case $P_2(25,19,7)$  yield
triples in the case $P_2$ via the quadratic pullback
\[  (A,B,C) \mapsto (B^A,B,C^2).\]
It follows from the linear rigidity and \cite{DR07} that
\[ \begin{array}{clcc}
    & \n(\cl(A),\cl(B),\cl(C)) &\mbox{condition}& \#\\
\hline
 1) &  \n(k_2,u_2,u_{6})=1 &&1 \\
 2) &\n(u_2,u_2,k_{3,2})=1 & & 1\\ 
3)&\n(k_2,u_2,h_{1a,1}')=1 & h_{1a,1}'^2= h_{1a,1}&2\\ 
4)&\n(k_2,u_2,h_{1b,1}')=1 & h_{1b,1}'^2= h_{b,1}&2 \\ 
5)& \n(k_2,u_2,h_{1}')=1 &h_{1}'^2=h_1&4\\ 
\end{array}\]

For an irreducible tuple $(g_1,g_2,g_3)$ in 1)
being no pullback we get for the braided tuple
\[ (g_2,g_1^{g_2},g_3)  \not \sim (g_1,g_2,g_3). \]
Otherwise there exists an element $h \in G_2$ such that
\[ {g_1^h=g_2}, \quad g_2^h =g_1^{g_2}, \quad g_3^h=g_3. \]
Thus
\[ (g_1,g_2,g_3)^{hg_2^{-1}}=(g_2,g_1,g_3^{hg_2^{-1}})\]
and
\[ (hg_2^{-1})^2 \in Z(G_2)=1. \]
This gives $h^2 =g_1g_2.$
Hence the pullback of
$ (g_1,hg_2^{-1},(g_1 h g_2^{-1})^{-1})$
is
\begin{eqnarray*}
(g_1,g_1^{(hg_2^{-1})^{-1}},(hg_2^{-1})^{2},(g_1hg_2^{-1})^{-2})&=&(g_1,g_2,1,(g_1hg_2^{-1})^{-2})=\\
(g_1,g_2,h^{-2})&=&(g_1,g_2,g_3).
\end{eqnarray*}

Thus in the case $(1)$ there is only one irreducible triple and
all irreducible tuples in the case $P_2(19,19,7)$ arise from pullbacks of
linearly rigid irreducible triples. 

For completeness we write down
the differential operators with  linearly rigid $G_2$-monodromy tuple. 
For this we
translate the construction from \cite{DR07} to the level of differential operators
and obtain 
formally selfdual operators $L$ with Riemann-scheme,
where
{\small
\[\begin{array}{ccl}
L&:=
&8\, \left( \vartheta-1 \right)  \left( \vartheta-2 \right)  \left( \vartheta-3 \right) 
 \left( 2\,\vartheta-1 \right)  \left( 2\,\vartheta-3 \right)  \left( 2\,\vartheta-5 \right) 
 \left( 2\,\vartheta-7 \right)\\ 
&& -8 x\, \left( 2\,\vartheta-5 \right)  \left( \vartheta-1 \right)  \left( \vartheta-2 \right) 
 \left( 2\,\vartheta-1 \right)  \left( 2\,\vartheta-3 \right)  \left( 8\,{\vartheta}^{2}-24\,\vartheta
+25-4(p^2+q^2+pq) \right) \\
&&+2 x^2  (\vartheta-1) (2 \vartheta-1) (2 \vartheta-3)
 (96 \vartheta^4-384 \vartheta^3+(720-96\,({q}^{2}+p^2+pq)) \vartheta^2+\\
&&(192\,({q}^{2}+p^2+pq)-672) \vartheta+
141+8\, \left( {p}^{2}-1+qp+{q}^{2} \right)  \left( 2\,{p}^{2}-15+2\,q
p+2\,{q}^{2} \right) )+\\
&& x^3 (2 \vartheta-1)  ( -256 \vartheta^6+768 \vartheta^5+(-1312+384\,({p}^{2}+{q}^{2}+qp)) \vartheta^4+\\
&&           (1344-768\,({p}^{2}+{q}^{2}+qp)) \vartheta^3-
 (160+32\, \left( 4\,({p}^{2}+q^2+pq)-21 \right)  \left( {p}^{2
}-1+qp+{q}^{2} \right)) \vartheta^2\\
&&+(32\, \left( 4\,({p}^{2}+q^2+pq)-9 \right)  \left( {p}^{2}-1+qp
+{q}^{2} \right)) \vartheta+\\
&& (64\,{p}^{2}{q}^{2} \left( {q}^{2}+2\,qp+{p}^{2} \right)-3) -8\, \left( 6
\,({p}^{2}+q^2+pq)-5 \right)  \left( {p}^{2}-1+qp+{q}^{2}
 \right) )+\\
&& 128 x^4  \vartheta (\vartheta - q) (\vartheta + q) (\vartheta - p) (\vartheta + p) (\vartheta+p + q) (\vartheta-p - q),
\end{array}\]
\[ {\mathcal R}(L)=\left\{\begin{array}{ccc}
                0 & 1 &\infty\\
         \hline
               1/2& 0& 0\\
             1& 0&   q\\
                3/2&1&  p\\
                2&1&   p + q\\
                5/2&1&  - p - q\\
                3& 2&  - q \\
                7/2& 2& - p
\end{array}\right\}. \]}

\subsection{The $P_3$ case}
In the case $P_3$  the list of possible Jordan forms of the local monodromies is as follows.

{\small \[\begin{array}{c|c|cccc}
  nr.   & \rk  &    & & &\\
\hline
 1   &  7   &   (  \Jordan(2),  \Jordan(2),  \Jordan(3))& (\Jordan(3),\Jordan(3),1) &(i,i,-1,1,i^{-1},i^{-1},-1) &\\
  2   &  7   &   (  \Jordan(2),  \Jordan(2),  \Jordan(3))& (\Jordan(3),\Jordan(3),1) & (-\Jordan(2),-\Jordan(2),\Jordan(3))\\
 3 &    7   &   (  \Jordan(2),  \Jordan(2),  \Jordan(3)) & (\Jordan(3),\Jordan(3),1)&(x,x,x^2,1,x^{-1},x^{-1},x^{-2})&\\
 4&    7   &   (  \Jordan(2),  \Jordan(2),  \Jordan(3)) & (\Jordan(3),\Jordan(3),1)&(\omega \Jordan(2),\omega^{-1}\Jordan(2),\omega,\omega^{-1},1)&\\

 5&    7   &   (  \Jordan(2),  \Jordan(2),  \Jordan(3)) &(-\Jordan(2),-\Jordan(2),1,1,1)&(i,i,-1,1,i^{-1},i^{-1},-1)\\
 6&    7   &   (  \Jordan(2),  \Jordan(2),  \Jordan(3)) &(-\Jordan(2),-\Jordan(2),1,1,1)&(-\Jordan(2),-\Jordan(2),\Jordan(3))\\
 7&    7   &   (  \Jordan(2),  \Jordan(2),  \Jordan(3)) &(-\Jordan(2),-\Jordan(2),1,1,1)&(x,x,x^2,1,x^{-1},x^{-1},x^{-2})&\\
 8&    7   &   (  \Jordan(2),  \Jordan(2),  \Jordan(3)) &(-\Jordan(2),-\Jordan(2),1,1,1)&(\omega \Jordan(2),\omega^{-1}\Jordan(2),\omega,\omega^{-1},1)&\\

 9&    7   &   (  \Jordan(2),  \Jordan(2),  \Jordan(3)) &(z,z,z^{-1},z^{-1},1,1,1)&(i,i,-1,1,i^{-1},i^{-1},-1)&\\
 10&    7   &   (  \Jordan(2),  \Jordan(2),  \Jordan(3)) &(z,z,z^{-1},z^{-1},1,1,1)&(-\Jordan(2),-\Jordan(2),\Jordan(3))&\\
 11&    7   &   (  \Jordan(2),  \Jordan(2),  \Jordan(3)) &(z,z,z^{-1},z^{-1},1,1,1)&(x,x,x^2,1,x^{-1},x^{-1},x^{-2})&\\
 12&    7   &   (  \Jordan(2),  \Jordan(2),  \Jordan(3)) &(z,z,z^{-1},z^{-1},1,1,1)&(\omega \Jordan(2),\omega^{-1}\Jordan(2),\omega,\omega^{-1},1)&\\
\end{array}\]}

Applying the functor
\[ M_\phi= \MT_{\L(\phi^{-1},1)} \circ \MC_\phi \circ \MT_{\L(\phi,1)} \circ \MC_{\phi^{-1}} ,\]
where 
\[ \phi=\left\{ \begin{array}{ccc}
                    i & 1),5),9) \\
                     -1&2),6),10)\\
                     x & 3),7),11)\\
                    \omega& 4),8),12)
                \end{array} \right.,\]
we obtain an orthogonally rigid local system of rank $5$  with the following tuple of Jordan forms of the local monodromies.
The contradiction of rank being $5$ and having Jordan form of type $ (\Jordan(3),\Jordan(3))$
in the cases (1)-(4) shows their nonexistence (reducibility).
 
{ \small\[ \begin{array}{c|c|cccc}
  nr.   & \rk  &    & & & \\
\hline
 1   &  5   &   (i,i,1,-i,-i)& (\Jordan(3),\Jordan(3)) &(  \Jordan(3),-1,-1) & red.\\
  2   &  5   &  (1,  -\Jordan(2),  -\Jordan(2)) & (\Jordan(3),\Jordan(3)) & \Jordan(5)& red.\\
 3 &    5   &  (x,x,1,x^{-1},x^{-1}) & (\Jordan(3),\Jordan(3))&(x^2,  \Jordan(3),x^{-2})& red. \\
 4&    5   &    (\omega,\omega,1,\omega^{-1},\omega^{-1}) & (\Jordan(3),\Jordan(3))&(\omega, \Jordan(3),\omega^{-1})& red.\\
\end{array}\]
\[ \begin{array}{c|c|cccc}
 5&    5   &   (i,i,1,-i,-i) &(-\Jordan(2),-\Jordan(2),1)&(  \Jordan(3),-1,-1)\\
 6&    5   &   (1,  -\Jordan(2),  -\Jordan(2)) &(-\Jordan(2),-\Jordan(2),1)&\Jordan(5)&\\
 7&    5   &  (x,x,1,x^{-1},x^{-1}) &(-\Jordan(2),-\Jordan(2),1)&(x^2,  \Jordan(3),x^{-2})&\\
 8&    5   &     (\omega,\omega,1,\omega^{-1},\omega^{-1}) &(-\Jordan(2),-\Jordan(2),1)&(\omega, \Jordan(3),\omega^{-1})&\\

 9&    5   &   (i,i,1,-i,-i) &(z,z,z^{-1},z^{-1},1)&(  \Jordan(3),-1,-1)&z^4\neq 1\\
 10&    5   &   (1,  -\Jordan(2),  -\Jordan(2)) &(z,z,z^{-1},z^{-1},1)&\Jordan(5)&\\
 11&    5   &   (x,x,1,x^{-1},x^{-1}) &(z,z,z^{-1},z^{-1},1)&(x^2,  \Jordan(3),x^{-2})&\\
 12&    5   &   (\omega,\omega,1,\omega^{-1},\omega^{-1}) &(z,z,z^{-1},z^{-1},1)&(\omega, \Jordan(3),\omega^{-1})&\\
\end{array}\]}

Using the isomorphism $ \La^2 \SP_4 \cong \SO_5 $ and the Scott Formula
we get
{\small \[ \begin{array}{c|c|cccc}
  nr.   & \rk  &    & & &\\
\hline
 5&    4   &   (i,1,1,-i) &  (-  \Jordan(2),1,1)  &(i  \Jordan(2),-i  \Jordan(2))\\
 6&    4   &   (-  \Jordan(2),1,1) &  (-  \Jordan(2),1,1)&-\Jordan(4)&\\
 7&    4   &  (x,1,1,x^{-1}) &  (-  \Jordan(2),1,1)&\pm (x\Jordan(2),x^{-1}  \Jordan(2))&\\
 8&    4   &     (\omega,1,1,\omega^{-1}) & (-  \Jordan(2),1,1)&\pm (\omega \Jordan(2),\omega^{-1}  \Jordan(2))&\\

 9&    4   &   (i,1,1,-i) &(z,1,1,z^{-1})&(i  \Jordan(2),-i  \Jordan(2)&\\
 10&    4   &   (-  \Jordan(2),1,1) &(z,1,1,z^{-1})&-\Jordan(4)&\\
 11&    4   &   (x,1,1,x^{-1}) &(z,1,1,z^{-1})&\pm (x  \Jordan(2),x^{-1}  \Jordan(2))&\\
 12&    4   &   (\omega,1,1,\omega^{-1}) &(z,1,1,z^{-1})&\pm (\omega \Jordan(2),\omega^{-1}  \Jordan(2))&\\
\end{array}\]}

Applying $\MC_{-1}$ we  obtain an orthogonally rigid local system of rank $3$ or $4$  with the following tuple of Jordan forms
of the local monodromies.
{ \small \[\begin{array}{c|c|cccc}
  nr.   & \rk  &    & & &\\
\hline

 5&    4   &   (i,1,1,-i) &  (  \Jordan(3),1)  &(i  \Jordan(2),-i  \Jordan(2))\\
 6&    3   &   \Jordan(3) &   \Jordan(3)&\Jordan(3)&\\
 7&    4   &  (-x,1,1,-x^{-1}) &  (  \Jordan(3),1)&\mp (x\Jordan(2),x^{-1}  \Jordan(2))&\\
 8&    4   &     (-\omega,1,1,-\omega^{-1}) & (  \Jordan(3),1)&\mp (\omega \Jordan(2),\omega^{-1}  \Jordan(2))&\\

 9&    4   &   (i,1,1,-i) &(-z,1,1,-z^{-1})&(i  \Jordan(2),-i  \Jordan(2)&\\
 10&    3   &     \Jordan(3) &(-z,1,-z^{-1})&\Jordan(3)&\\
 11&    4   &   (-x,1,1,-x^{-1}) &(-z,1,1,-z^{-1})&\mp (x  \Jordan(2),x^{-1}  \Jordan(2))&\\
 12&    4   &   (-\omega,1,1,-\omega^{-1}) &(-z,1,1,-z^{-1})&\mp (\omega \Jordan(2),\omega^{-1}  \Jordan(2))&\\
\end{array}\]}

Via the isomorphisms 
$ \sym^2 \SP_2 \cong \SO_3$ and $\SL_2\otimes \SL_2 \cong \SO_4 $
we can decompose it into linearly rigid irreducible 
local systems $\L_1$ and $\L_2$ of rank $2$ with the following tuple of Jordan forms.
The conditions for a product of eigenvalues not being $1$ is due to the irreducibility condition from
the Scott Formula.

{\small \[ \begin{array}{c|c|ccccc}
  nr.   & \L_i  &    & & & \\
\hline

 5&  \L_1   &   (\zeta_8,\zeta_8^{-1}) & \Jordan(2)  &\pm\Jordan(2) \\
  &  \L_2  & (\zeta_8,\zeta_8^{-1}) &  \Jordan(2)&     (i,-i)\\

 6&  \L_1=\L_2   &    \Jordan(2) &    \Jordan(2)&-\Jordan(2)&\\
 7&  \L_1    &  (\tilde{x},\tilde{x}^{-1})  &  \Jordan(2) & \Jordan(2) &\tilde{x}^2=-x\\
  &  \L_2     &(\tilde{x},\tilde{x}^{-1}) &  \Jordan(2)&   \pm  (x,x^{-1})&\tilde{x}^2=-x\\
8&  \L_1     &  (\zeta_{12},\zeta_{12}^{-1})    &\Jordan(2)  &   \pm \Jordan(2) &\\
 &\L_2& (\zeta_{12},\zeta_{12}^{-1})&\Jordan(2)& \pm  (\omega,\omega^{-1})&\\
9& \L_1   &    (\zeta_8,\zeta_8^{-1})  &(\tilde{z},\tilde{z}^{-1}) & \Jordan(2) & \tilde{z}^2=-z & \zeta_8^{\pm 1}\tilde{z}^{\pm 1}\neq 1  \\
& \L_2   &  (\zeta_8,\zeta_8^{-1}) &(\tilde{z},\tilde{z}^{-1})&  (i,-i)& \tilde{z}^2=-z&  \zeta_8^{\pm 1}\tilde{z}^{\pm 1} i^{\pm 1}\neq 1\\
 10& \L_1=\L_2   &     \Jordan(2) &(\tilde{z},\tilde{z}^{-1})&  \Jordan(2)& \tilde{z}^2=-z\\
 11&  \L_1   &  (\tilde{x},\tilde{x}^{-1}) & (\tilde{z},\tilde{z}^{-1}) & \Jordan(2) &\tilde{x}^2=-x, \tilde{z}^2=-z & \tilde{x}^{\pm 1}\tilde{z}^{\pm 1} \neq 1 \\
 &  \L_2   & (\tilde{x},\tilde{x}^{-1}) & (\tilde{z},\tilde{z}^{-1}) &(x,x^{-1})&\tilde{x}^2=-x, \tilde{z}^2=-z& \tilde{x}^{\pm 1}\tilde{z}^{\pm 1} x^{\pm 1}\neq 1 \\
 12&  \L_1    &   (\zeta_{12},\zeta_{12}^{-1})  &(\tilde{z},\tilde{z}^{-1}) &  \Jordan(2) & \tilde{z}^2=-z& \zeta_{12}^{\pm 1}\tilde{z}^{\pm 1}\neq 1 \\
 &\L_2&(\zeta_{12},\zeta_{12}^{-1})  &(\tilde{z},\tilde{z}^{-1}) &  (\omega,\omega^{-1}) & \tilde{z}^2=-z& \zeta_{12}^{\pm 1}\tilde{z}^{\pm 1} \omega^{\pm 1}\neq 1 
\end{array}\]}

Thus there exist at most $4$
 orthogonally rigid
local systems having $G_2$-monodromy with the same local monodromies.
Computing the normalized structure constant of the reduced monodromy tuple
 we show
that there exist at most $2$ such local systems.

$$ \begin{array}{|ccc|} 
\hline
n(u_2,k_{2,1},k_{2,2})&=&\frac{3q-1}{q}\\
n(u_2,k_{2,1},k_{3,1})&=&2\\
n(u_2,k_{2,1},h_{1b})&=&\left\{ \begin{array}{cc}
                          2 & o(h_{1b}) \mid (q-1)/2\\
                          0 & o(h_{1b}) \nmid (q-1)/2
                       \end{array}\right.\\
n(u_2,h_{1a},k_{2,2})&=&\left\{ \begin{array}{cc}
                          2 & o(h_{1a}) \mid (q-1)/2\\
                          0 & o(h_{1a}) \nmid (q-1)/2
                       \end{array}\right.\\
n(u_2,h_{1a},k_{3,1})&=&\left\{ \begin{array}{cc}
                          2 & o(h_{1a}) \mid (q-1)/2\\
                          0 & o(h_{1a}) \nmid (q-1)/2
                       \end{array}\right.\\
n(u_2,h_{1a},h_{1b})&=&\left\{ \begin{array}{cc}
                          2 & o(h_{1a}) \mid (q-1)/2, \;o(h_{1b}) \mid (q-1)/2\\
                          0 & o(h_{1b}) \nmid (q-1)/2
                       \end{array}\right.\\

\hline
\end{array}$$

Replacing $q$ by $q^2$ we see that there are at most $2$ such local systems in the case $P_3$
with the same local monodromies.
The existence follows from the construction of the corresponding
differential operators. 
Let
\[ P_3:=L_{2c+1/2}*_H L_{-2c+1/2} *_H   (L_3 \ten (\Lambda^2(L_{1/2}*_H  (L_0 \ten(L_1\ten L_2)))))        \]         
where
\[ {\mathcal{R}}(L_3)=\left\{\begin{array}{ccc}
                             0 & 1 & \infty \\
  \hline
                             1 & 0 & -1 
\end{array}
                     \right\},\quad {\mathcal{R}}(L_{\alpha})=\left\{\begin{array}{ccc}
                             0 & 1 & \infty \\
  \hline
                             0 & -\alpha & \alpha 
\end{array}
                     \right\}, \alpha \in \{1/2,\pm 2c+1/2\},\]
\[ L_0=\vartheta-1/2, \quad {\mathcal{R}}(L_0)=\left\{\begin{array}{ccc}
                             0 & 1 & \infty \\
  \hline
                             1/2 & 0 & -1/2 
\end{array}
                     \right\}\]
and
\[\begin{array}{ccl}
 L_1&:=&4\, \left( \vartheta-c \right)  \left( \vartheta+c \right) +z \left( -
8\,{\vartheta}^{2}-4\,\vartheta-1-4\,{b}^{2}+4\,{c}^{2} \right) +{z}^{
2} \left( 2\,\vartheta+1 \right) ^{2}\\
L_2&:=&4\, \left( \vartheta-c \right)  \left( \vartheta+c \right) +z \left( -
8\,{\vartheta}^{2}-4\,\vartheta-1-4\,{b}^{2}+20\,{c}^{2} \right) +{z}^
{2} \left( 2\,\vartheta+1+4\,c \right)  \left( 2\,\vartheta+1-4\,c
 \right) 
\end{array}\]
with
{\small \[ {\mathcal{R}}(L_1)=\left\{\begin{array}{ccc}
                             0 & 1 & \infty \\
                             \hline
                             c & b & 1/2 \\
                             -c &-b& 1/2
              \end{array}
                     \right\},\quad  {\mathcal{R}}(L_2)=\left\{\begin{array}{ccc}
                             0 & 1 & \infty \\
                             \hline
                             c & b & 2c+1/2 \\
                             -c &-b& -2c+1/2
              \end{array}
                     \right\}.\]}

Then
\[\begin{array}{ccl}
P_{3} &=&16\,{\vartheta}^{2} \left( \vartheta-2 \right) ^{2} \left( \vartheta-1
 \right) ^{3}-\\
&&8\,x{\vartheta}^{2} \left( 2\,\vartheta-1 \right) 
 \left( \vartheta-1 \right) ^{2} \left( 4\,{\vartheta}^{2}-4\,
\vartheta+8\,{b}^{2}+5-24\,{c}^{2} \right) +\\
&&4\,{x}^{2}{\vartheta}^{3}
 ( 24\,{\vartheta}^{4}+(38-288 {c
}^{2}+64\,{b}^{2}){\vartheta}^{2}+16\,{b}^{2}+64\,{b}^{4}-144\,{c}^{2}+
7+576\,{c}^{4}-384\,{c}^{2}{b}^{2}) -\\
&&2\,{x}^{3} ( 2\,
\vartheta+1 )  \left( 2\,\vartheta+1+4\,c \right)  \left( 2\,
\vartheta+1-4\,c \right)  ( 4\,{\vartheta}^{4}+8\,{\vartheta}^{3}
+11\,{\vartheta}^{2}+8\,{\vartheta}^{2}{b}^{2}-56\,{\vartheta}^{2}{c}^
{2}+\\
&&7\,\vartheta-56\,\vartheta\,{c}^{2}+8\,\vartheta\,{b}^{2}+64\,{c}^
{4}+2+4\,{b}^{2}-36\,{c}^{2}-64\,{c}^{2}{b}^{2}) +\\
&&{x}^{4}
 \left( \vartheta+1 \right)  \left( 2\,\vartheta+3-4\,c \right) 
 \left( 2\,\vartheta+1-4\,c \right)  \left( \vartheta+1-4\,c \right) 
 \left( \vartheta+1+4\,c \right)\cdot \\
&& \left( 2\,\vartheta+3+4\,c \right) 
 \left( 2\,\vartheta+1+4\,c \right) 
\end{array}\]
with
{\small \[ {\mathcal{R}}(P_{3})=\left\{\begin{array}{ccc}
                             0 & 1 & \infty \\
                             \hline
                             0 & 0 & 1 \\
                             0 & 1& 2c+1/2\\
                             1& 2& 2c+3/2\\
                             1 &1/2+2b& 4c+1 \\
                             1& 3/2+2b&-4c+1\\
                             2& 3/2-2b&-2c+3/2\\
                             2& 1/2-2b&-2c+1/2
                        \end{array}
                     \right\}.\]}
 
Thus if we replace $b$ by $b+1/2$ (or $c$ by $c+1/2$) in the construction
we get the same local monodromies for $P_3(b)$ and $P_3(b+1/2)$.
However if $L_1$ is reducible, i.e. if $\pm b \pm c+1/2 \in \ZZ$
or $L_1(b) \sim L_1(b+1/2)$, i.e.  $2b\in 1/2+\ZZ$,  then there is only one
$P_3$ with the given local monodromies.

Since $\Lambda^2$ yields an operator of degree $14$
we get that $P_3$ has  $G_2$-monodromy.

\subsection{The case $P_5$}
We start with the possible list of Jordan forms of the local monodromies of orthogonally rigid quadruples with
 $G_2$-monodromy.

{ \small \[\begin{array}{c|c|ccccc}
  nr.   & \rk  &    & & &\\
\hline
1    &  7   &   (  \Jordan(2),  \Jordan(2),1_3)&  (  \Jordan(2),  \Jordan(2),1_3)&  (\omega 1_3,\omega^{-1}1_3,1) & (\Jordan(3),\Jordan(3),1)&\\
 2   &  7   &   (  \Jordan(2),  \Jordan(2),1_3)&  (  \Jordan(2),  \Jordan(2),1_3)&  (\omega 1_3,\omega^{-1}1_3,1) & (-\Jordan(2),-\Jordan(2),1_3)&\\
3    &  7   &   (  \Jordan(2),  \Jordan(2),1_3)&  (  \Jordan(2),  \Jordan(2),1_3)&  (\omega 1_3,\omega^{-1}1_3,1) & (x1_2,x^{-1}1_2,1_3)&\\
     &&&&&x^3\neq 1
\end{array}\]}
Applying the functor
\[ M_\phi=\MT_{\L(1,1,\phi^{-1})} \circ  \MC_{\phi \omega^{-1}} \circ \MT_{\L(1,1,\phi \omega)} \circ \MC_{\phi^{-1} \omega} \circ \MT_{\L(1,1,\omega^{-1})},\]
where 
\[ \phi=\left\{ \begin{array}{ccc}
                    1 & 1) \\
                     -1&2)\\
                    x & 3)\\
                    \end{array} \right.,\]
we obtain an orthogonally rigid local system of rank $4$ or $5$ with the following tuple of Jordan forms of the
local monodromies.

{\small \[\begin{array}{c|c|cccccc}
  nr.   & \rk  &    & & &\\
\hline
1    &  4   &   (  \Jordan(2),  \Jordan(2))&  (  \Jordan(2),  \Jordan(2))&  (\Jordan(3),1) & (\omega,\omega^{-1},1_2)&\\
2    &  5   &   (  \Jordan(2),  \Jordan(2),1)&  (  \Jordan(2),  \Jordan(2),1)& (  -\Jordan(2),  -\Jordan(2),1)  &(\omega,\omega^{-1},1_3) &\\
3    &  5   &   (  \Jordan(2),  \Jordan(2),1)&  (  \Jordan(2),  \Jordan(2),1)&  (x 1_2,x^{-1}1_2,1) & (\omega,\omega^{-1},1_3)&\\
\end{array}\]}

Applying the functor
\[ M_\phi=\MT_{\L(1,\phi^{-1},1)}\circ \MC_{\phi} \circ \MT_{\L(1,\phi,\phi^{-1})} \circ \MC_{\phi^{-1}} \circ  \MT_{\L(1,1,\phi)},\]
where 
\[ \phi=\left\{ \begin{array}{ccc}
                    1 & 1) \\
                     -1&2)\\
                    x & 3)\\
                                   \end{array} \right.,\]
we obtain an orthogonally rigid local system of rank $4$  with the following tuple of Jordan forms of the
local monodromies.

{\small \[\begin{array}{c|c|ccccc}
  nr.   & \rk  &    & & &\\
\hline
  1  &  4   &   (  \Jordan(2),  \Jordan(2))&  (  \Jordan(2),  \Jordan(2))&  (\Jordan(3),1) & (\omega,\omega^{-1},1,1)&\\
 2   &  4   &   (  \Jordan(2),  \Jordan(2))&  (  -\Jordan(2),  -\Jordan(2))& ( \Jordan(3),1)  &(\omega,\omega^{-1},1,1) &\\
  3  &  4   &   (  \Jordan(2),  \Jordan(2))&  (x1_2, x^{-1}1_2)& ( \Jordan(3),1)   & (\omega,\omega^{-1},1,1)&
\end{array}\]}

Via the isomorphism 
\[ \SL_2\otimes \SL_2 = \SO_4 \]
we can decompose it into linearly rigid irreducible 
local systems $\L_1$ and $\L_2$ of rank $2$ with the following tuple of Jordan forms of the
local monodromies.

{ \small \[\begin{array}{c|cccc|ccccc}
  nr.   & \L_1  &    & & &\L_2\\
\hline
  1  &    \Jordan(2) &  1_2  &   \Jordan(2)  &  (\omega,\omega^{-1})&    1_2 &  \Jordan(2)&  \Jordan(2) & (\omega,\omega^{-1})\\
     &    \Jordan(2) &  1_2  &   \Jordan(2)  &  -(\omega,\omega^{-1})&   1_2 &  \Jordan(2)&  \Jordan(2) & -(\omega,\omega^{-1})\\
 2  &    \Jordan(2) &  1_2  &   \Jordan(2)  &   (\omega,\omega^{-1})&    1_2 &  -\Jordan(2)&  \Jordan(2) & (\omega,\omega^{-1})\\
     &    \Jordan(2) &  1_2  &   \Jordan(2)  &  -(\omega,\omega^{-1})&   1_2 &  -\Jordan(2)&  \Jordan(2) & -(\omega,\omega^{-1})\\
3  &    \Jordan(2) &  1_2  &   \Jordan(2)  &  (\omega,\omega^{-1})&   1_2 &  (x,x^{-1})&  \Jordan(2) & (\omega,\omega^{-1})\\
     &    \Jordan(2) &  1_2  &   \Jordan(2)  &  -(\omega,\omega^{-1})&   1_2 & (x,x^{-1})&  \Jordan(2) & -(\omega,\omega^{-1})&x^6\neq 1
\end{array}\] } 

>From the discussion above we know that there exist at most $2$
 orthogonally rigid
local systems having $G_2$-monodromy with the same local monodromies.
The generic character table of the group $G_2(q)$ shows
that there exists at most $1$ such local system.
The normalized structure constant of the reduced monodromy tuple
gives 
\[\begin{array}{|cccccccc|} 
\hline
n(u_2,u_2,k_3,u_3)&=&1&n(u_2,u_2,k_3,u_4)&=&n(u_2,u_2,k_3,u_5)&=&0\\
n(u_2,u_2,k_3,k_{2,1})&=&1 &&&n(u_2,u_2,k_3,h_{1a})&=&1\\
\hline
\end{array}\]

There are infinitely many such quadruples due to the positions $(1,t,0,\infty) $ of the singularities.
 Those with singularities at $1,-1,0,\infty$ arise from quadratic pullbacks of 
 $P_1$-cases with the following tuples of Jordan forms.
{\small  \[\begin{array}{c|ccc}
 nr.   &   &    &  \\
\hline
1&  (  \Jordan(2),  \Jordan(2),1_3)& (-\omega 1_2,-\omega^{-1}1_2,\omega,\omega^{-1},1)& (\Jordan(3),-\Jordan(3),-1)\\
2& (  \Jordan(2),  \Jordan(2),1_3)& (-\omega 1_2,-\omega^{-1}1_2,\omega,\omega^{-1},1)& (i\Jordan(2),-i\Jordan(2),-1,-1,1)\\
3& (  \Jordan(2),  \Jordan(2),1_3)& (-\omega
1_2,-\omega^{-1}1_2,\omega,\omega^{-1},1)& (y,-y^{-1},-1,1,-1,-y,y^{-1})\\
&&&y^2=x\\
\end{array}\]}
Therefore all quadruples exist independent of their singular locus 
since shifting the position $-1$ to $t$ does not effect the group properties.
Thus the orthogonally rigid local systems
with $G_2$-monodromy in the case $P_5$ are uniquely determined by their
local monodromy and singular locus.\\

This finishes the proof of Thm~\ref{mainthm}. \Endproof

\appendix
\section{Generic character tables and structure constants}\label{Char}

 Let $\C=(C_1,\ldots,C_{r+1})$ be a 
 tuple of  conjugacy classes  of a group $G$
 and 
 \[ \Sigma(\C)=\{ \si \in G^{r+1} \mid \si_i \in C_i, \si_1 \cdots \si_{r+1}=1\}.\]
 Then the {\it normalized structure constant} $n(\C)$ is defined as 
 \[ n(\C)=\frac{\mid \Sigma(\C) \mid}{\mid Inn(G)\mid}. \]
The following result is well known (cf.~\cite[Chap. I, Thm.~5.8]{MalleMatzat}):

\begin{prop}\label{Propstructconst} Let $G$ be a finite group,  let 
$\Irr(G)$ denote the set of irreducible characters  of $G$ and let 
$\C=(C_1,\ldots,C_{r+1})$ be a tuple of conjugacy classes of $G$ with representatives
$\si_1,\ldots,\si_{r+1}.$  Then 
 \[ n(\C) =
    \frac{\mid Z(G) \mid \cdot \mid G\mid ^{r-1}}{\prod_i \mid C_G(\si_i) \mid}
     \sum_{\chi \in \Irr(G)} \frac{\prod_i \chi(\si_i)} {\chi(1)^{r-1}}. \] 
\end{prop}
 In order to find an upper bound for the number of local systems
 with the same tuple of local monodromies we use reduction modulo $l$
 and derive the bound from the normalized structure constant:\\
 
Let $G$ be a reductive  algebraic group defined over $\ZZ$ which is 
an irreducible subgroup of $\GL_n$ (e.g. $G_2\leq \GL_7$) and let 
  $\C=(C_1,\ldots,C_{r+1})$  be a  tuple of conjugacy classes 
  in $G.$ Consider the map
$$\pi: C_1\times \cdots\times  C_{r+1}\to G, (g_1,\ldots,g_{r+1})\mapsto g_1\cdots g_{r+1}$$ 
and let $X:=\pi^{-1}(1)$ (with $1\in G$ the neutral element).  
The variety $X$ decomposes into irreducible components 
$X_1,\ldots,X_k.$ 
The following result is the content of  \cite{Katz96}, Lemma 5.9.3, and will be useful below:

\begin{lem} \label{5.9.3}Let $R$ be a subring of $\CC$ which is finitely generated as a $\ZZ$-algebra. Then there exists 
an $N\in \NN_{>0}$  such that for any prime number $\ell$ which does not divide 
$N,$  there exists a finite extension $K_\nu$ of $\QQ_\ell$ with valuation ring 
$O_\nu$ and an isomorphism $\iota:\CC\to \bar{\QQ}_\ell$ under which 
$R$ is mapped into $O_\nu.$ 
\end{lem}

The idea of the proof is as follows: Using Noether normalization, $R$ is an integral 
extension of $\ZZ[\frac{1}{N}][x_1,\ldots,x_{r+1}],$ where $x_1,\ldots,x_{r+1}$ are algebraically
independent. By the axiom of choice, for any algebraically independent set $\{y_1,\ldots,y_{r+1}\}\subseteq \ZZ_\ell$ (where
$\ell$ does not divide $N$), there exists an isomorphism $\iota: \CC\to \bar{\QQ}_\ell$ which
maps $x_i$ to~$y_i, \, i=1,\ldots,{r+1}.$

 Lemma~\ref{5.9.3} implies that
  for any tuple $\C$ of conjugacy classes there exists an $M \in \NN_{>0}$
 such that for any prime number $\ell$ which does not divide 
$M,$  there exists a finite extension $K_\nu$ of $\QQ_\ell$ with valuation ring 
$O_\nu$ such that $C_1\times \cdots \times C_{r+1}$ and  $X=\pi^{-1}(1)$ is defined over $O_\nu.$ 
 Similarly, for any 
 $\g=(g_1,\ldots,g_{r+1})\in X$ there exists an $N\in \NN_{>0}$
 such that for any prime number $\ell$ which does not divide 
$N,$  there exists a finite extension $K_\nu$ of $\QQ_\ell$ with valuation ring 
$O_\nu$ such that the coefficients of all elements of $\g$ are contained in $O_\nu.$ 
  Hence, for almost all $\ell$  we find $\nu\mid \ell$ such we can reduce the entries of $\g$  
   modulo the valuation ideal $m_\nu\subseteq \OO_\nu.$ In this way we obtain 
   the {\it reduced tuple} $\bar{\g}=(\bar{g}_1,\ldots,\bar{g}_{r+1})\in (\bar{C}_1,\ldots, \bar{C}_{r+1}),$ where 
   $\bar{C}_i$ is the conjugacy class of $\bar{g}_i$ in $G(\FF_q),$ where
   $\FF_q=O_\nu/m_\nu.$  
  For positive natural numbers $k,$ let 
  ${\C}(q^k)$
denote the tuple of
  conjugacy classes
 of $\bar{g}_1,\ldots, \bar{g}_{r+1}$   in the group $G(\FF_{q^k}).$ 

\begin{thm}\label{nC}
  Suppose that $G$ is an irreducible  simple algebraic subgroup of $\GL_n(\CC)$ defined over 
$\ZZ$ and suppose that there exists an $s\in \NN{>0}$ such that 
$$ \sup_{q}(\lim_k \lfloor n(\C(q^k)) \rfloor)=s,$$ where the supremum is taken over all 
prime powers $q$ which are cardinalities of  the residue fields
of  $\nu$ as above.
 Then, up to diagonal $G(\CC)$-conjugation,  
there exist at most $s$ tuples 
$$\g_i:=(g_{i,1},\ldots,g_{i,r+1})\in C_1\times \cdots \times C_{r+1} \quad (i=1,\ldots, s)$$ 
with $g_{i,1}\cdots g_{i,r+1}=1$ and such that the generated subgroup  
$\langle g_{i1},\ldots, g_{i,r+1}\rangle $ is irreducible. \end{thm}

\begin{proof} 
Assume that there exist  $t>s$ different equivalence 
classes (w.r. to diagonal $G(\CC)$-conjugation) of  tuples 
$$\g_i=(g_{i,1},\ldots,g_{i,r+1})\in C_1\times \cdots \times C_{r+1} \quad (i=1,\ldots, t)$$ 
with $g_{i,1}\cdots g_{i,r+1}=1$ and such that the generated subgroup  
$\langle  g_{i1},\ldots, g_{i,r+1} \rangle $ is irreducible. 
We have the following two cases:

{\it Case1:} The tuples $\g_i=(g_{i,1},\ldots,g_{i,r+1})\, (i=1,\ldots,t)$ lie in $t$ different irreducible components $X_i$ of $X.$ 
By Lemma~\ref{5.9.3}, for almost all $\ell$ there exists a finite extension $K_\nu$ of $\QQ_\ell$  such that 
$\g_i\in X_i(O_\nu)\, (i=1,\ldots,t).$ If $\ell>>0$ and $k>>0,$  then the reductions modulo 
$m_\nu$ of the components $X_i$ remain different. Hence  reduction modulo the 
maximal ideal $m_\nu$ of $O_\nu$   leads to $t$ 
different equivalence classes (under diagonal conjugation with elements in $G(\FF_q)$)
$\bar{\g}_i\in (\bar{C}_1,\ldots, \bar{C}_{r+1}),$ contrary to $t>s=\sup_{q}(\lim_k \lfloor n(\C(q^k))\rfloor ) .$  

{\it Case 2:} Two of the tuples, say $\g_1$ and $\g_2$,
lie in the same irreducible component $X_1.$ 
Since  $\langle \g_1 \rangle $ is irreducible  the  $G(\CC)$-stabilizer of $\g_1
\in C_1\times \cdots \times C_{r+1}$ 
under diagonal conjugation is equal to the centralizer of 
$\langle \g_1 \rangle$ and hence coincides with the (finite) centre  $Z(G)$ of $G.$  
  This implies that  the dimension of the component $X_1$ of $X$ with $\g_1 \in X_1$ is $\geq \dim G.$ 
  Therefore, by the assumption in Case~2, 
$\dim X_1> \dim G$ and, by dimension reasons, there exist infinitely many $G(\CC)$-orbits $V_{j}\, (j\in J)$ 
in $X_1.$ Pick $u>s$ different orbits $V_{1},\ldots,V_{u}$ and representatives 
$$v_k \in V_{k}\quad (k=1,\ldots,u).$$ 
Suppose that the $V_{1},\ldots,V_{u}$ are defined over $R,$ where we see
$R$ as a subring of $O_\nu$ ($\nu |\ell$) as above.
We claim that for $\ell>>0,$ the reductions modulo-$\nu$ are different. This can be seen
inductively as follows: The orbits are (quasi-)affine varieties inside an ambient
affine space $\AA^s.$ Pick functions $f_j$ in the vanishing ideals 
of $V_{j}$ with the property that for 
$j\not= j',$ there exists $v_{j'}\in V_{j'}$ with $f_j(v_{j'})\not= 0.$ Extending scalars and 
assuming $\ell$ large enough, we
can assume   that the functions $f_j$ and the $v_{j}$ are defined over $R$ and hence over
$O_\nu.$ If $f_j(v_{j'})$ is algebraic, then for almost all $\ell$ the inequality 
$f_j(v_{j'})\not= 0$ will hold modulo $\nu$ for all pairs of 
$j,j'$ where $j\not=j'.$   If $f_j(v_{j'})$ is transcendent, then with the 
freedom to choose  the isomorphism $\iota:\CC\to \bar{\QQ}_\ell$ (see the remark
following Lemma~\ref{5.9.3}) in a way that
the inequality 
$f_j(v_{j'})\not= 0$ will hold modulo $\nu$  for all pairs of 
$j,j'$ where $j\not=j'.$  Therefore, the orbits 
remain different modulo $\nu$ and hence 
$$ \sup_{q}(\lim_k \lfloor n(\C(q^k))\rfloor )\geq u,$$  a contradiction to $u>s.$
\end{proof}

\begin{rem}\label{remchevie} Recall that there are character tables for $G(\FF_q)$ which compute the character
values of the elements of $G(\FF_q)$ as functions depending on $q,$ the so called
{\it generic character tables}. For groups with small Lie rank, these generic character 
tables are implemented in \cite{CHEVIE}, especially, the case $G=G_2$, cf. \cite{ChangRee} and \cite[Anhang B]{Hiss},   can be found there. 
Using the generic character table of $G_2(q)$ we can determine $n(\bar{\C}(q^k))$ 
and also  $ \sup_q(\lim_k \lfloor n(\C(q^k))\rfloor )$ in many cases.  
 \end{rem}

\begin{rem}\label{ChaRee}
We give an overview of  the class representatives $c_j$ in $G_2(q)$ taken from
 Chang and Ree. 
In order to determine $\lim_k \lfloor n(\C(q^k))\rfloor$ we can assume  that the eigenvalues of all class representatives of
$C_1,\ldots, C_{r+1}$
are in $\FF_q$.
Otherwise we replace $q$ by $q^4$.
The generic character table depends on the congruence of $q \mod 12$.
Thus we can also assume that $q \equiv 1 \mod 12$.
We list in the
notation of Chang and Ree \cite{ChangRee}
for class representative $c_j$ having eigenvalues in $\FF_q$ the corresponding
Jordan forms.
 The order of the centralizer of $c_j$ in  $G \in \{G_2, O_7 ,\GL_7\}$  is a polynomial in $q$ of degree 
 $d_G:=\dim C_{G(\overline{\FF_q})}(g_j)$.

{\small $$ \begin{array}{c|c|ccc|ccc}
\mbox{class rep.} & \mbox{Jordan form} & d_{G_2} &  d_{O_7} &  d_{\GL_7}  & \mbox{conditions}\\
\hline
1 & 1& 14 & 21& 49&\\
 u_1 & (J(2),J(2),1,1,1)& 8& 13 &29&\\
 u_2 & (J(3),J(2),J(2))&6&9&19&\\
u_3 &(J(3),J(3),J(1))& 4&7&17&\\
u_4 &(J(3),J(3),J(1))&4&7&17&\\
u_5 &(J(3),J(3),J(1))&4&7&17&\\
u_6 & J(7)&2 &3 &7&\\
\hline
k_2 & (-1_4,1_3)& 6&9&25 & &\\
k_{2,1}&(-J(2),-J(2),1,1,1)&4&7&17\\
k_{2,2}&(-J(2),-J(2),J(3))&4&5&11\\
k_{2,3}&(-J(3),-J(1),J(3))&2&3&9\\
k_{2,4}&(-J(3),-J(1),J(3))&2&3&9\\
\hline
k_3 & (\omega 1_3,1,\omega^{-1}1_3 )& 8 & 9 & 19 &\omega^3=1\\
k_{3,1}&(\omega J(2),\omega^{-1}J(2),\omega,\omega^{-1},1)&4&5&11&\\
k_{3,2}&(\omega J(3),\omega^{-1}J(3),1 )&2&3&7&\\
k_{3,3,i}&(\omega J(3),\omega^{-1}J(3),1 )&2&3&7&\;i=1,2,\;k_{3,3,1}^{-1} \sim k_{3,3,2} \\
\hline 
h_{1a}&(x,x,x^{-1},x^{-1},1,1,1)&4 &7&17&x^{q-1}=1,\,x^2 \neq 1\\
h_{1a,1}&(xJ(2),x^{-1}J(2),J(3))&2 &3&7\\
h_{1b}&(x,x,x^2,1,x^{-1},x^{-1},x^{-2})&4 &5&11&x^{q-1}=1,\;x^3\neq 1,x^4 \neq 1 \\
&(i,i,-1,1,-1,i^{-1},i^{-1}) &4 &5&13& \\
\hline 
h_{1b,1}&(xJ(2),x^{-1}J(2),x^2,x^{-2},1)&2&3&7 \\
&(iJ(2),i^{-1}J(2),-1,-1,1)&2&3&9 \\

h_1 &(x,y,xy,1,(xy)^{-1},y^{-1},x^{-1})&2&3&7& x^{q-1}=y^{q-1}=1 \\
&& &&&\textrm{pairw. diff. eigenvalues }\\                
    &(x,-1,-x,1,-x^{-1},-1,x^{-1})&2&3&9 &&\\
\end{array}
$$}
\end{rem}

\end{document}